\documentclass[a4paper,11pt]{amsart}
\usepackage{amsmath}
\usepackage{amssymb}
\usepackage{amsthm}
\usepackage[all]{xy}
\usepackage[latin1]{inputenc}        
\usepackage[dvips]{graphics}
\usepackage[dvips]{graphicx}
\usepackage{amscd}
\usepackage{mathrsfs}
\usepackage[mathcal]{eucal}
\usepackage{fullpage}
\usepackage{float}
\author{Francesco Polizzi}
\address{Dipartimento di Matematica, Università della Calabria, Via Pietro Bucci,
87036 Arcavacata di Rende (CS), Italy.}
\email{polizzi@mat.unical.it}
\title[Numerical properties of isotrivial fibrations]{Numerical properties of
isotrivial fibrations}

\newtheorem{inizio}{Lemma}[section]
\newtheorem{theorem}[inizio]{Theorem}
\newtheorem{corollary}[inizio]{Corollary}
\newtheorem{proposition}[inizio]{Proposition}
\newtheorem{remark}[inizio]{Remark}
\newtheorem{lemma}[inizio]{Lemma}

\newtheorem{definition}[inizio]{Definition}

\newtheorem{problem}[inizio]{Problem}

\newtheorem*{teoA}{Theorem A}
\newtheorem*{teoB}{Theorem B}

\newtheorem*{corollary-s}{Corollary C}
\theoremstyle{definition}
\newtheorem{example}[inizio]{Example}

\newcommand{\lr}{\longrightarrow}

\newcommand{\mO}{\mathcal{O}}
\newcommand{\mZ}{\mathbb{Z}}
\newcommand{\mE}{\mathsf{E}}

\newcommand{\St}{\left( -2+ \sum_{i=1}^r \left(1- \frac{1}{ \;n_i} \right) \right)}

\newcommand{\cF}{\mathfrak{c}(F)}
\newcommand{\type}{\big( \frac{q_1}{n_1}, \ldots,
\frac{q_r}{n_r}  \big)}

\newcommand{\mg}{\mathfrak{g}}

\newcommand{\si}[2]{ \frac{1}{#1} (1,#2)}

\newcommand{\ssi}[6]{
\hline
$\si{#1}{#2}$ & $#3$ & $\si{#1}{#4}$  & $#5$ & $#6$ \\
}

\begin{document}



\abstract In this paper we investigate the numerical properties of
relatively minimal isotrivial fibrations $\varphi \colon X \lr C$,
where $X$ is a smooth, projective surface and $C$ is a curve. In
particular we prove that, if $g(C) \geq 1$ and $X$ is neither ruled
nor isomorphic to a quasi-bundle, then $K_X^2 \leq 8 \chi(\mO_X)-2$;
this inequality is sharp and if equality holds then $X$ is a minimal
surface of general type whose canonical model has precisely two
ordinary double points as singularities. Under the further
assumption that $K_X$ is ample, we obtain $K_X^2 \leq 8
\chi(\mO_X)-5$ and the inequality is also sharp. This improves
previous results of Serrano and Tan.

\endabstract

\maketitle

\section{Introduction}

One of the most useful tools in the study of algebraic surfaces is the
analysis of \emph{fibrations}, that is morphisms
 with connected fibres from a surface $X$
onto a curve $C$. When all smooth fibres of a fibration $\varphi
\colon X \lr C$ are isomorphic to each other, we call $\varphi$
 an \emph{isotrivial fibration}.
As far as we know, there is hitherto no systematic study of
minimal models of isotrivial fibrations; the aim of the
present paper is to shed
some light on this problem. \\
\indent A smooth, projective surface $S$ is called a \emph{standard
isotrivial fibration} if there exists a finite group $G$, acting
faithfully on two smooth projective curves $C_1$ and $C_2$, so that
$S$ is isomorphic to the minimal desingularization of $T:=(C_1
\times C_2)/G$, where $G$ acts diagonally on the product. When the
action of $G$ is free, then $S=T$ is called a \emph{quasi-bundle}.
These surfaces have been investigated in \cite{Se90}, \cite{Se96},
\cite{Ca00}, \cite{BaCaGr08}, \cite{Pol07}, \cite{CarPol07},
\cite{MiPol08}, \cite{BaCaGrPi08}. A monodromy argument shows that
every isotrivial fibration $\varphi \colon X \lr C$ is birationally
isomorphic to a standard one $($\cite[Section $2$]{Se96}$)$; this
means that there exist $T=(C_1 \times C_2)/G$, a birational map $T
\dashrightarrow X$
 and an isomorphism $C_2/G \lr C$ such that the diagram
\begin{equation*}
\xymatrix{
T \ar@{-->}[r]  \ar[d] &
 X \ar[d]^{\varphi} \\
 C_2/G \ar[r]^{\cong} & C}
\end{equation*}
commutes.\\
\indent If $\lambda \colon S \lr T=(C_1 \times C_2)/G$ is any
standard isotrivial fibration, the two projections $\pi_1 \colon C_1
\times C_2 \lr C_1$, $\pi_2 \colon C_1 \times C_2 \lr C_2$ induce
two morphisms $\alpha_1 \colon S \lr C_1/G$, $\alpha_2 \colon S \lr
C_2/G$,
 whose smooth fibres are isomorphic to $C_2$ and $C_1$, respectively.
 Moreover $q(S)=g(C_1/G)
 +g(C_2/G)$. If $S$ is a quasi-bundle,
 then all singular fibres of $\alpha_1$ and $\alpha_2$ are
 multiple of smooth curves. Otherwise, $T$ contains some cyclic quotient
 singularities, and the invariants $K_S^2$ and $e(S)$ can be computed
 in terms of the number and
 type of such singularities.
Moreover  the corresponding fibres of $\alpha_1$ and
 $\alpha_2$ consist of an irreducible curve, called the central component,
 with at least two Hirzebruch-Jung strings attached. Assume that a fibre $F$
 of $\alpha_1$ (or $\alpha_2$) contains
 exactly $r$ such strings, of type $\frac{1}{n_1}(1, \, q_1), \ldots,
\frac{1}{n_r}(1, \, q_r)$, respectively; therefore
we say that $F$ is of type $\type$. \\
\indent Now set $\mg:=g(C_1)$ and consider a reducible fibre $F$ of $\alpha_2 \colon S \lr C_2/G$.
If $g(C_1/G) =0$ then it may happen that the central component of
$F$ is a $(-1)$-curve; in this case we say that $F$ is a
$(-1)$-\emph{fibre in genus} $\mg$. Moreover, if $g(C_2/G) \geq 1$
then the central
components of $(-1)$-fibres of $\alpha_2$ are the unique $(-1)$-curves on $S$. \\
\indent Our first result provides a method to construct standard
isotrivial fibrations with arbitrary many $(-1)$-fibres.
\begin{teoA}[see Theorem \ref{existence}]
Let $\mathcal{S}:=\big\{ \frac{q_1}{n_1}, \ldots, \frac{q_r}{n_r}
\big\} $ be a finite set of rational numbers, with $(n_i, q_i)=1$, such
that $\sum_{i=1}^r \frac{q_i}{n_i}=1$. Set $n:=\emph{l.c.m.}(n_1,
\ldots, n_r)$. Then for any $\mathfrak{q} \geq 0$ there exists a
standard isotrivial fibration $\lambda \colon S \lr T=(C_1 \times
C_2)/G$ such that the following holds.
\begin{itemize}
\item[$(i)$] $\emph{Sing}(T)=n \times \frac{1}{n_1}(1,q_1)+ \cdots +
 n \times \frac{1}{n_r}(1, q_r);$
\item[$(ii)$] the singular fibres of
 $\,\alpha_2 \colon S \lr C_2/G$ are exactly $\;n\;$ $(-1)$-fibres, all of type
 $\type;$
\item[$(iii)$] $q(S)=\mathfrak{q}$.
\end{itemize}
\end{teoA}
Our second result deals with the ``geography" of (minimal models of)
isotrivial fibrations. It is straightforward to prove that every
quasi-bundle $S$ satisfies $K_S^2 = 8 \chi(\mO_S)$. In \cite{Se96}
Serrano extended this result, showing that any isotrivial fibred
surface $X$ satisfies $K_X^2 \leq 8 \chi(\mO_X)$; his proof is based
on the properties of the projective cotangent bundle
$\mathbb{P}(\Omega^1_X)$. Exploiting the fact that every isotrivial
fibration is birationally isomorphic to a standard one, we obtain
the following strengthening of Serrano's theorem. We want to
emphasize that our methods involves mostly arguments of
combinatorial nature, and it is very different from Serrano's one.
\begin{teoB}[see Theorem \ref{arbitrary-iso}]
Let $\varphi \colon X \lr C$ be any relatively minimal isotrivial fibration,
with $X$ non ruled and $g(C) \geq 1$. If
 $X$ is not isomorphic to a quasi-bundle, we have
\begin{equation} \label{u}
K_X^2 \leq 8 \chi(\mO_X)-2
\end{equation}
and if equality holds then $X$ is a minimal surface of general
type whose canonical model
has precisely two ordinary double points as singularities. \\
Moreover, under the further assumption that $K_X$ is ample, we
have
\begin{equation} \label{v}
K_X^2 \leq 8 \chi(\mO_X)-5.
\end{equation}
Finally, both inequalities \eqref{u} and \eqref{v} are sharp.
\end{teoB}
We do not know whether Theorem B remains true
if one drops the assumption $g(C) \geq 1$. \\      \\
\indent Let us now illustrate the structure of the paper and give
a brief account of how the results are achieved.   \\
\indent In Section \ref{sec:prel} we review some of the standard
facts about group actions on Riemann surfaces and cyclic quotient
singularities; in particular we recall the Riemann existence theorem
and the Hirzebruch-Jung resolution in terms of continued fractions;
furthermore, we make some computations that will be
used in Section \ref{sec:relatively minimal}.\\
\indent In Section \ref{sec:standard-iso} we summarize the basic
properties of standard isotrivial fibrations. This section is
strongly inspired by Serrano's papers \cite{Se90} and \cite{Se96},
but our approach is different. In particular, we provide some
results on the singular locus of $T$ which one could not obtain by
means of Serrano's techniques (Corollaries \ref{integers} and \ref{nodes}).\\
\indent In Section \ref{sec:non-minimal} we start the analysis of
the case where $S$ is not a minimal surface. In particular we give
necessary and sufficient conditions ensuring that a reducible fibre
$F$ is a $(-1)$-fibre (Proposition \ref{(-1)-fibre}), and this
allows us to prove Theorem A. \\
\indent In Section \ref{sec:relatively minimal} we look more closely
at the relative minimal model $\hat{\alpha}_2 \colon
 \widehat{S} \lr C_2/G$ of $\alpha_2 \colon S \lr C_2/G$.
The main step is to define, for any reducible fibre $F$ of
$\alpha_2$, an invariant $\delta(F) \in \mathbb{Q}$ such that
\begin{equation} \label{aaaaa}
 K_{\widehat{S}}^2 = 8 \chi (\mO_{\widehat{S}})- \sum_{F \,
\textrm{reducible} } \delta(F).
\end{equation}
We also obtain a combinatorial classification of $(-1)$-fibres. When
$\mg=0$, the so-called Riemenschneider's duality between the
$HJ$-expansions of $\frac{n}{q}$ and $\frac{n}{n-q}$ implies
$\delta(F)=0$. If $\mg \geq 1$ one has instead $\delta(F)>2$ for all
reducible fibres $F$, with precisely three exceptions that we
describe in detail (Corollary \ref{fibres}). Using these facts,
together with relation \eqref{aaaaa} and some identities on
continued fractions shown in Section 1, we prove Theorem B. In
particular, the proof of inequality \eqref{v} uses the computer
algebra program \verb|GAP4|, whose database includes all groups of
order less that $2000$, with the exception of $1024$ (see
\cite{GAP4}). However, the computer can be replaced either by
(tedious) hand-made computations or by
the Atlas of Finite Groups (\cite{CCPW}). \\
\indent In Appendix A we classify all possible types of
$(-1)$-fibres for $\mg=1, \, 2, \, 3$; we also relate this
classification to those given by Kodaira (when
$\mg=1$) and Ogg (when $\mg=2$). \\
\indent Finally, in Appendix B we provide a list of all the cyclic
quotient singularities $\frac{1}{n}(1, \, q)$ and their numerical
invariants, for $ 2 \leq n \leq 14$. We hope that this will help
the reader to check our computations.

\bigskip
$\mathbf{Notations \; and \; conventions}$. All varieties in this
article are defined over $\mathbb{C}$. If $S$ is a projective,
non-singular surface $S$ then $K_S$ denotes the canonical class,
$p_g(S)=h^0(S, \; K_S)$ is the \emph{geometric genus},
$q(S)=h^1(S, \; K_S)$ is the \emph{irregularity} and
$\chi(\mathcal{O}_S)=1-q(S)+p_g(S)$ is the \emph{Euler
characteristic}. We denote by $\textrm{kod}(S)$ the Kodaira
dimension of $S$ and we say that $S$ is \emph{ruled} if
$\textrm{kod}(S)=-\infty$. For every finite group $G$,
the notation $G=G(|G|, \, \ast)$  indicates the label of $G$ in the
$\verb|GAP4|$ database of small groups. For instance, $D_4=G(8, \,
3)$ means that $D_4$ is the third in the list of groups of order
$8$. If $x \in G$ the conjugacy class of $x$ is denoted by $\textrm{Cl}(x)$.
If $x$ and $y$ are conjugate in $G$ we write $x \sim_G y$.
The commutator of $x$ and $y$  is defined as $[x, \, y]=xyx^{-1}y^{-1}$.
The derived subgroup of $G$ is denoted by $[G, \, G]$.
\bigskip

$\mathbf{Acknowledgements.}$ This research started
 when the author was visiting professor at the
University of Bayreuth (September-December 2007), supported by the
\emph{DFG Forschergruppe ``Klassifikation algebraischer Flächen und
kompakter komplexer Mannigfaltigkeiten"}. He wishes to thank I.
Bauer, F. Catanese, E. Mistretta and R. Pignatelli  for many
enlightening conversations and helpful suggestions. Moreover, he is
indebted to the organizers of the semester ``Groups in Algebraic
Geometry"  (especially F. Catanese and R. Pardini) and to the
``Centro Ennio de Giorgi" (University of Pisa, Italy) for the
invitation and hospitality during September 2008. Finally, he
expresses his gratitude to S. L. Tan for sending him the paper
\cite{Tan96}, which contains some results related to those obtained
in the present work.
\newpage

\section{Preliminaries} \label{sec:prel}

\subsection{Group actions on Riemann surfaces}

\begin{definition} \label{generating vect}
Let $G$ be a finite group and let
\begin{equation*}
\mathfrak{g}' \geq 0, \quad   m_r \geq m_{r-1} \geq \ldots \geq m_1
\geq 2
\end{equation*}
be integers. A \emph{generating vector} for $G$ of type
$(\mathfrak{g}' \; | \; m_1, \ldots ,m_r)$ is a $(2
\mathfrak{g}'+r)$-tuple of elements
\begin{equation*}
\mathcal{V}=\{g_1, \ldots, g_r; \; h_1, \ldots, h_{2\mathfrak{g}'}
\}
\end{equation*}
such that the following conditions are satisfied:
\begin{itemize}
\item the set $\mathcal{V}$ generates $G$;
\item the order of $g_i$ is equal to $m_i$;
\item $g_1g_2\cdots g_r \Pi_{i=1}^{\mathfrak{g}'} [h_i,h_{i+\mathfrak{g}'}]=1$.
\end{itemize}
If such a $\mathcal{V}$ exists, then $G$ is said to be
$(\mathfrak{g}' \; | \; m_1, \ldots ,m_r)$-\emph{generated}.
\end{definition}
\begin{remark} \label{abelian-gen}
If an \emph{abelian} group $G$ is $(\mathfrak{g}' \; | \; m_1,
\ldots ,m_r)$-generated then either $r=0$ or $r \geq 2$. Moreover if
$r=2$ then $m_1=m_2$.
\end{remark}
For convenience we make abbreviations such as $(4 \;| \; 2^3, 3^2)$
for $(4 \; | \; 2,2,2,3,3)$ when we write down the type of the
generating vector $\mathcal{V}$.
\begin{proposition}[Riemann Existence Theorem] \label{riemann ext}
A finite group $G$ acts as a group of automorphisms of some
compact Riemann surface $C$ of genus $\mathfrak{g}$ if and only if
there exist integers $\mathfrak{g}' \geq 0$ and $m_r \geq m_{r-1}
\geq \ldots \geq m_1 \geq 2$ such that $G$ is $(\mathfrak{g}'\;
|\; m_1, \ldots, m_r)$-generated, with generating vector
$\mathcal{V}=\{g_1, \ldots, g_r; \; h_1, \ldots,
h_{2\mathfrak{g}'} \}$, and the Riemann-Hurwitz relation holds:
\begin{equation} \label{riemanhur}
2\mathfrak{g}-2=|G| \left(
2\mathfrak{g}'-2+\sum_{i=1}^r\bigg(1-\frac{1}{\;m_i} \bigg)
 \right).
\end{equation}
If this is the case, $\mathfrak{g}'$ is the genus of the quotient
Riemann surface $D:=C/G$ and the $G$-cover $C \lr D$ is branched in
$r$ points $P_1, \ldots, P_r$ with branching numbers $m_1, \ldots,
m_r$, respectively. In addition, the subgroups $\langle g_i \rangle$
and their conjugates provide all the nontrivial stabilizers of the
action of $G$ on $C$.
\end{proposition}
In the situation of Proposition \ref{riemann ext} we shall say that
$G$ \emph{acts in genus} $\mathfrak{g}$ \emph{with signature}
$(\mathfrak{g}' \, | \, m_1, \ldots, m_r)$. We refer the reader to
[Br90, Section 2], [Bre00, Chapter 3],
\cite{H71} and [Pol07, Section 1] for more details. \\
Now let $C$ be a compact Riemann surface of genus $\mathfrak{g} \geq
2$ and let $G \subseteq \textrm{Aut}(C)$. For any $h \in G$ set
$H:=\langle h \rangle$ and define the set of fixed points of $h$ as
\begin{equation*}
\textrm{Fix}_C(h)=\textrm{Fix}_C(H):=\{x \in C \; |\; hx=x \}.
\end{equation*}
For our purposes it is also important to take into account how an
automorphism acts in a neighborhood of each of its fixed points. We
follow the exposition of \cite[pp.17, 38]{Bre00}. Let $\mathscr{D} \subset \mathbb{C}$
be the unit disk and $h \in \textrm{Aut}(C)$ of order $m >1$ such
that $hx =x$ for a point $x \in C$.
 Then there is a unique
 primitive complex $m$-th root of unity $\xi$ such that any lift of
 $h$ to $\mathscr{D}$ that fixes a point in $\mathscr{D}$ is
 conjugate to the transformation $z \lr \xi \cdot z$ in
 $\textrm{Aut}(\mathscr{D})$. We write $\xi_x(h)=\xi$ and we
 call $\xi^{-1}$ the \emph{rotation constant} of $h$ in $x$.
Then for each integer $ 1 \leq q \leq m-1$ such that $(m,\,q)=1$ we define
\begin{equation*}
\textrm{Fix}_{C,q}(h)=\{x \in \textrm{Fix}_C(h) \; | \;
\xi_x(h)=\xi^{q} \},
\end{equation*}
that is the set of fixed points of $h$ with rotation constant
$\xi^{-q}$. Clearly, we have
\begin{equation*}
\textrm{Fix}_C(h)=\biguplus_{\substack{ 1 \leq q \leq m-1 \\
(m, \, q)=1}}\textrm{Fix}_{C,q}(h).
\end{equation*}
\begin{proposition} \label{fixed-points}
Assuming that we are in the situation of Proposition
\emph{\ref{riemann ext}}, let $h \in G^{\times}$ be of order $m$,
$H=\langle h \rangle$ and $(m, \, q)=1$. Then
\begin{equation*} \label{formula-per-fix-tot}
|\emph{Fix}_{C}(h)|=|N_G(H)| \cdot \sum_{\substack{1 \leq i \leq r
\\ m|m_i
\\ H \; \sim_G \;  \langle g_i^{m_i/m} \rangle }} \frac{1}{\; m_i}
\end{equation*}
and
\begin{equation*} \label{formula-per-fix}
|\emph{Fix}_{C,q}(h)|=|C_G(h)| \cdot \sum_{\substack{1 \leq i \leq r
\\ m|m_i
\\ h \; \sim_G \;  g_i^{m_iq/m} }} \frac{1}{\; m_i} ~.
\end{equation*}
\end{proposition}
\begin{proof}
See \cite[Lemma 10.4 and 11.5]{Bre00}.
\end{proof}

\begin{corollary} \label{cor:fix}
Assume that $h \; \sim_G \; h^q$. Then
$|\emph{Fix}_{C,1}(h)|=|\emph{Fix}_{C,q}(h)|$.
\end{corollary}

\subsection{Surface cyclic quotient singularities and
Hirzebruch-Jung  resolutions} \label{HJ-res} Let $n$ and $q$ be
coprime natural numbers with $1 \leq q \leq n-1$,
 and let $\xi_n$  be a primitive $n$th root of unity.
Let us consider the action of the cyclic group $\mathbb{Z}_n=\langle
\xi_n \rangle$ on $\mathbb{C}^2$  defined by $\xi_n \cdot
(x,y)=(\xi_nx, \xi_n^qy)$. Then the analytic space
$X_{n,q}=\mathbb{C}^2 / \mathbb{Z}_n$ contains a cyclic quotient
singularity of type $\frac{1}{n}(1,q)$. Denoting by $q'$ the unique
integer $1 \leq q' \leq n-1$ such that $qq' \equiv 1$ (mod $n$), we
have $X_{n_1,q_1} \cong X_{n, q}$ if and only if $n_1=n$ and either
$q_1=q$ or $q_1=q'$. The exceptional divisor on the minimal
resolution $\tilde{X}_{n,q}$ of $X_{n,q}$ is a $HJ$-string
 (abbreviation of Hirzebruch-Jung string), that is to say, a
 connected union $\mE=\bigcup_{i=1}^k Z_i$ of smooth rational curves $Z_1, \ldots, Z_k$ with
 self-intersection $\leq -2$, and ordered linearly so that $Z_i
 Z_{i+1}=1$ for all $i$, and $Z_iZ_j=0$ if $|i-j| \geq 2$.
More precisely, given the continued fraction
\begin{equation} \label{cont-frac}
\frac{n}{q}=[b_1,\ldots,b_k]=b_1-
                                \cfrac{1}{b_2 -\cfrac{1}{\dotsb
                                 - \cfrac{1}{\,b_k}}}, \quad b_i\geq 2 ~,
\end{equation}

the dual graph of $\mE$ is  {\setlength{\unitlength}{1.1cm}
\begin{center}
\begin{picture}(1,0.5)
\put(0,0){\circle*{0.2}} \put(1,0){\circle*{0.2}}
\put(0,0){\line(1,0){1}} \put(-0.3,0.2){\scriptsize $-b_1$}
\put(0.70,0.2){\scriptsize $-b_2$} \put(2,0){\circle*{0.2}}
\put(1,0){\line(1,0){0.2}} \put(1.3,0){\line(1,0){0.15}}
\put(1.55,0){\line(1,0){0.15}} \put(1.8,0){\line(1,0){0.2}}
\put(3,0){\circle*{0.2}} \put(2,0){\line(1,0){1}}
\put(1.70,0.2){\scriptsize $-b_{k-1}$} \put(2.70,0.2){\scriptsize
$-b_k$}
\end{picture}         \hspace{2.5cm}
\end{center}
}

\vspace{.5cm}

\noindent (see \cite[Chapter II]{Lau71}).
Moreover
\begin{equation} \label{q-prime}
\frac{n}{q}=[b_1, \ldots, b_k] \quad \textrm{if and only if} \quad  \frac{n}{q'}=[b_k, \ldots,
b_1].
\end{equation}
In particular a rational double point of type $A_n$ corresponds to the cyclic
quotient singularity $\frac{1}{n+1}(1,n)$. A point of type
$\frac{1}{2}(1,1)$ is called an \emph{ordinary double point} or
a \emph{node}. For any $1 \leq s \leq
k$ set $\frac{n_s}{q_s}:=[b_1, \ldots, b_s]$; then $\big
\{\frac{n_s}{q_s} \big \}$ is called the sequence of
\emph{convergents} of the continued fraction \eqref{cont-frac}.
Its terms satisfy the recursive relation
\begin{equation} \label{eq:recursion}
\frac{n_s}{q_s}=\frac{b_s n_{s-1}-n_{s-2}}{b_s q_{s-1}-q_{s-2}},
\end{equation}
where $n_{-1}=0, \; n_0=1, \; q_{-1}=-1, \; q_0=0$ (see Appendix to \cite{OW77}).
\begin{proposition} \label{decreasing}
The sequence $\big\{ \frac{n_s}{q_s} \big\}$ is strictly decreasing,
in fact
\begin{equation} \label{eq:decreasing-1}
\frac{n_{s-1}}{q_{s-1}}-\frac{n_s}{q_s}=\frac{1}{q_{s-1}q_s}.
\end{equation}
Consequently, the sequence $\big\{ \frac{q_s}{n_s} \big\}$ is
strictly increasing, in fact
\begin{equation} \label{eq:decreasing-2}
\frac{q_s}{n_s}-\frac{q_{s-1}}{n_{s-1}}=\frac{1}{n_s n_{s-1}}.
\end{equation}
\end{proposition}
\begin{proof}
Using \eqref{eq:recursion} we can write
\begin{equation*}
\begin{split}
n_{s-1}q_s-n_sq_{s-1}&=n_{s-1}(b_s q_{s-1}-q_{s-2})-(b_s
n_{s-1}-n_{s-2})q_{s-1} \\
&=n_{s-2}q_{s-1}-n_{s-1}q_{s-2}= \ldots =n_1q_2-n_2q_1 \\
&=b_1b_2-(b_1b_2-1)=1,
\end{split}
\end{equation*}
so both \eqref{eq:decreasing-1} and \eqref{eq:decreasing-2} follow
at once.
\end{proof}

\begin{definition} \label{numbers}
Let $x$ be a cyclic quotient singularity of type $\frac{1}{n}(1,q)$
 and let $\mE$ be the corresponding HJ-string. If
 $\frac{n}{q}=[b_1, \ldots, b_k]$, we write $\mE \colon [b_1, \ldots, b_k]$ and
 we set
\begin{equation*}
\begin{split}
\ell_x=\ell({\mE})=\ell \bigg(\frac{q}{n} \bigg) &:=k,\\
h_x=h({\mE})=h \bigg(\frac{q}{n} \bigg) &:=2- \frac{2+q+q'}{n}-\sum_{i=1}^k (b_i-2), \\
e_x=e({\mE})=e \bigg(\frac{q}{n} \bigg) &:=k+1-\frac{1}{n}, \\
B_x=B({\mE})=B \bigg(\frac{q}{n} \bigg) &:= 2e_x - h_x = \frac{1}{n}
(q + q') + \sum_{i=1}^k b_i.
\end{split}
\end{equation*}
\end{definition}

\begin{remark} \label{geq 3}
We have
\begin{equation*}
\ell \bigg(\frac{q}{n} \bigg)= \ell \bigg(\frac{q'}{n} \bigg), \; \;
h \bigg(\frac{q}{n} \bigg)=h \bigg(\frac{q'}{n} \bigg), \;\; e
\bigg(\frac{q}{n} \bigg)=e \bigg(\frac{q'}{n} \bigg), \;\; B
\bigg(\frac{q}{n} \bigg)=B \bigg(\frac{q'}{n} \bigg).
\end{equation*}
Moreover $B \big(\frac{q}{n} \big) \geq 3$ and equality holds if and
only if $\frac{q}{n} = \frac{1}{2}$.
\end{remark}

For the reader's convenience, we listed in the Appendix B the
cyclic quotient singularities $\frac{1}{n}(1,\,q)$ and the
corresponding
 values of $h \big(\frac{q}{n} \big)$ and $B \big(\frac{q}{n} \big)$ for all
 $ 2 \leq n \leq 14$.

\begin{proposition} \label{B-increasing}
Let $\frac{n_s}{q_s}, \; \frac{n_t}{q_t}$ be two convergents of the
continued fraction $\frac{n}{q}=[b_1, \ldots, b_k]$, with $s \geq t$.
Then
\begin{equation*} \label{eq:B-increasing}
B \bigg(\frac{q_s}{n_s} \bigg) - B \bigg(\frac{q_t}{n_t} \bigg)
 \geq s-t
\end{equation*}
and equality holds if and only if $s=t$.
\begin{proof}
It is sufficient to prove that $B \big(\frac{q_s}{n_s} \big) - B
\big(\frac{q_{s-1}}{n_{s-1}} \big) > 1$. In fact we have
\begin{equation*}
B \bigg(\frac{q_s}{n_s} \bigg) - B \bigg(\frac{q_{s-1}}{n_{s-1}}
\bigg)=\frac{q_s}{n_s} - \frac{q_{s-1}}{n_{s-1}} +
\frac{q'_s}{n_s} - \frac{q'_{s-1}}{n_{s-1}} + b_s,
\end{equation*}
that is, using \eqref{eq:decreasing-2},
\begin{equation*}
B \bigg(\frac{q_s}{n_s} \bigg) - B \bigg(\frac{q_{s-1}}{n_{s-1}}
\bigg)> \frac{1}{n_sn_{s-1}}
-\frac{q_{s-1}'}{n_{s-1}}+b_s > b_s-1 \geq 1.
\end{equation*}
\end{proof}
\end{proposition}

\begin{corollary} \label{cor-B-increasing}
Let $\frac{n}{q}=[b_1, \ldots, b_k]$ and let $c \in \mathbb{N}$ be
such that $b_1 \geq c$. Then
\begin{equation*}
B \bigg(\frac{q}{n} \bigg) \geq B \bigg(\frac{1}{c} \bigg)=
c+\frac{2}{c}
\end{equation*}
and equality holds if and only if  $\frac{q}{n}=\frac{1}{c}$.
\end{corollary}
\begin{proof}
Setting $s=k$ and $t=1$ in Proposition \ref{B-increasing} we obtain
\begin{equation*}
B \bigg( \frac{q}{n}\bigg) \geq B \bigg( \frac{1}{b_1}\bigg)= b_1 + \frac{2}{b_1} \geq
c + \frac{2}{c}=B \bigg( \frac{1}{c}\bigg)
\end{equation*}
and equality holds if and only if $k=1$ and $c=b_1$.
\end{proof}

There is a duality between the $HJ$-expansions of $\frac{n}{q}$ and
$\frac{n}{n-q}$, which comes from the Riemenschneider's point
diagram (\cite[p. 222]{Rie74}). It basically says that if
$\frac{q}{n} \neq \frac{1}{2}$ then there exist nonnegative integers
$k_1, \ldots, k_t$, $l_1, \ldots, l_{t-1}$ such that
\begin{equation} \label{duality}
\begin{split}
\frac{n}{q}&=[(2)^{k_1}, \; l_1+3, \; (2)^{k_2},  \ldots,
(2)^{k_{t-1}}, \; l_{t-1}+3,
\; (2)^{k_t}], \\
\frac{n}{n-q}&=[k_1+2, \; (2)^{l_1}, \; k_2+3, \ldots,  k_{t-1}+3,
\; (2)^{l_{t-1}}, \;
 k_t+2],
\end{split}
\end{equation}
where $(2)^k$ means the constant sequence with $k$ terms equal to
$2$, in particular the empty sequence if $k=0$. It is important to
notice that both the $k_i$ or the $l_j$ may actually be equal to
zero; for instance, the case $q=1$ (i.e. $\frac{n}{1}=[n]$,
$\frac{n}{n-1}=[(2)^{n-1}]$) is obtained by setting $t=2$, $k_1=0$,
$l_1=n-3$, $k_2=0$. From a more geometric point of view, if $N$
denotes a free abelian group of rank $2$, then \eqref{duality}
reflects the duality between the oriented cone $\sigma_{n,\,q}
\subset N_{\mathbb{R}}$ associated to $\frac{n}{q}$ and the oriented
cone $\sigma_{n,\,n-q}$ associated to $\frac{n}{n-q}$ (see
\cite{NePo08}). Now let us exploit Riemenschneider's duality in
order to obtain some results on continued fractions that will be
used in the proof of Proposition \ref{delta-magg-0}.

\begin{proposition} \label{duality-2}
We have
\begin{equation*}
B \bigg(\frac{q}{n} \bigg)+B \bigg(\frac{n-q}{n} \bigg)=3
\sum_{i=1}^t (k_i+1)+ 3 \sum_{i=1}^{t-1} (l_i+1).
\end{equation*}
\end{proposition}
\begin{proof}
Using \eqref{duality} we obtain
\begin{equation*}
\begin{split}
B \bigg(\frac{q}{n} \bigg)&=\frac{q}{n} + \frac{q'}{n}+ 2
\sum_{i=1}^t k_i + \sum_{i=1}^{t-1} (l_i+3), \\
B \bigg(\frac{n-q}{n} \bigg)&=\frac{n-q}{n}+\frac{(n-q)'}{n} +
\sum_{i=1}^t (k_i+3) + 2\sum_{i=1}^{t-1} l_i -2.
\end{split}
\end{equation*}
Combining these relations and using $(n-q)'=n-q'$ we conclude the
proof.
\end{proof}

\begin{proposition} \label{prop-fraz-cont-1}
Let $n, \, q$ be positive, coprime integers and let $a$
be such that $qq'=1 + a n$.
Assume moreover that
\begin{equation*}
[(2)^{k_1}, \; l_1+3, \; (2)^{k_2}, \ldots, (2)^{k_{t-1}}, \;
l_{t-1}+3, \; (2)^{k_t}]= \frac{n}{n-q'}
\end{equation*}
for some non negative integers $k_1, \ldots, k_t$, $\,l_1, \ldots,
l_{t-1}$. Then we have
\begin{align}
& [ k_1+2, \; (2)^{l_1}, \; k_2+3, \ldots, k_{t-1}+3, \;
(2)^{l_{t-1}},
\; k_t+3] = \frac{n+q}{a+q'} \label{fraz-cont-11} \\
& [ k_1+2, \; (2)^{l_1}, \; k_2+3, \ldots, k_{t-1}+3, \;
(2)^{l_{t-1}}]
  =  \frac{q}{a}. \label{fraz-cont-12}
\end{align}
\end{proposition}
\begin{proof}
Using  \eqref{q-prime} and  \eqref{duality} we can write
\begin{equation} \label{kt+3}
\begin{split}
[k_t+3, \; (2)^{l_{t-1}}, \ldots, \; (2)^{l_1}, \; k_1+2] & = 1+
[k_t+2,
\; (2)^{l_{t-1}}, \ldots, (2)^{l_1}, \; k_1+2] \\
&=1+ \frac{n}{(q')'}=1+\frac{n}{q}=\frac{n+q}{q}.
\end{split}
\end{equation}
Since $q\cdot (a+q') \equiv 1$ (mod $n+q$) and $1 \leq a+q' <n+q$,
from
\eqref{q-prime} we obtain \eqref{fraz-cont-11}. \\
Now we have
\begin{equation*}
\begin{split}
\frac{n}{q}&=[k_t+2, \; (2)^{l_t-1}, \; k_{t-1}+3, \ldots, k_2+3, \;
(2)^{l_1}, \; k_1+2]\\ &= k_t+2-[(2)^{l_t-1}, k_{t-1}+3, \ldots,
k_2+3, \;(2)^{l_1}, \; k_1+2]^{-1},
\end{split}
\end{equation*}
which implies
\begin{equation*}
[(2)^{l_t-1}, \; k_{t-1}+3, \ldots, k_2+3, \; (2)^{l_1}, \;
k_1+2]=\frac{q}{q(k_t+2)-n}.
\end{equation*}
Since $a \cdot (q(k_t+2)-n)) \equiv 1$ (mod $q$) and $1 \leq a <q$,
by using \eqref{q-prime} we obtain \eqref{fraz-cont-12}.
\end{proof}

\begin{proposition} \label{prop-fraz-cont-2}
With the notations of Proposition \emph{\ref{prop-fraz-cont-1}}, we
have
\begin{align}
B \bigg(\frac{n-q'}{n} \bigg)+ B \bigg( \frac{a+q'}{n+q} \bigg) & =
1-\frac{1+q^2}{n(n+q)} +3 \sum_{i=1}^t(k_i+1) +3 \sum_{i=1}^{t-1}(l_i+1),
\label{fraz-cont-21} \\
B \bigg( \frac{n-q'}{n} \bigg) + B \bigg( \frac{a}{q} \bigg) &
=-\frac{1+q^2+n^2}{nq} + 3 \sum_{i=1}^t(k_i+1) +3
\sum_{i=1}^{t-1}(l_i+1). \label{fraz-cont-22}
\end{align}
\end{proposition}
\begin{proof}
Write
\begin{align}
B \bigg( \frac{n-q'}{n} \bigg)&= \frac{n-q'}{n} + \frac{n-q}{n}+ 2
\sum_{i=1}^t k_i + \sum_{i=1}^{t-1} (l_i+3), \label{E1} \\
B \bigg( \frac{a+q'}{n+q} \bigg)&= \frac{a+q'}{n+q} +
\frac{q}{n+q} + \sum_{i=1}^t (k_i+3) + 2\sum_{i=1}^{t-1} l_i -1, \label{E2} \\
B \bigg( \frac{a}{q} \bigg) &= \frac{a}{q} + \frac{q(k_t+2)-n}{q} +
\sum_{i=1}^{t-1} (k_i+3) + 2\sum_{i=1}^{t-1} l_i -1. \label{E3}
\end{align}
Summing \eqref{E1} and \eqref{E2} we obtain \eqref{fraz-cont-21},
whereas summing \eqref{E1} and \eqref{E3} we obtain
\eqref{fraz-cont-22}.
\end{proof}

\section{Standard isotrivial fibrations} \label{sec:standard-iso}

In this section we summarize the basic properties of standard
isotrivial fibrations. Definition \ref{def-stand} and Theorem
\ref{Serrano} can be found in \cite{Se96}.

\begin{definition} \label{def-stand}
We say that a projective surface $S$ is a  \emph{standard isotrivial
fibration} if there exists a finite group $G$ acting faithfully on
two smooth projective curves $C_1$ and $C_2$ so that $S$ is
isomorphic to the minimal desingularization of $T:=(C_1 \times
C_2)/G$, where $G$ acts diagonally on the product. The two maps
$\alpha_1 \colon S \lr C_1/G$, $\alpha_2 \colon S \lr C_2/G$ will be
referred as the \emph{natural projections}. If $T$ is smooth then
$S=T$ is called a $\emph{quasi-bundle}$.
\end{definition}

\begin{remark} \label{no-loss-if-standard}
A monodromy argument shows that every isotrivial fibred surface $X$
is birationally isomorphic to a standard isotrivial fibration
$($\cite[Section $2$]{Se96}$)$.
\end{remark}

The stabilizer $H  \subseteq G$ of a point $y \in C_2$ is a cyclic
group (\cite[p. 106]{FK92}). If $H$ acts freely on $C_1$, then $T$
is smooth along the scheme-theoretic fibre of $\sigma \colon T \lr
C_2/G$
 over $\bar{y} \in C_2/G$, and this fibre consists of the curve $C_1/H$
 counted with multiplicity $|H|$. Thus, the smooth fibres of $\sigma$
 are all isomorphic to $C_1$. On the contrary, if $ x \in C_1$ is fixed
 by some non-zero element of $H$, then one has a cyclic quotient
 singularity  over the point $\overline{(x,y)} \in T$.
These observations lead to the following
 statement, which describes the singular fibres that can arise in a
 standard isotrivial fibration (see \cite[Theorem 2.1]{Se96}).
\begin{theorem} \label{Serrano}
Let $\lambda \colon S \lr T=(C_1 \times C_2)/G$ be a standard
isotrivial fibration and let us consider the natural projection
 $\alpha_2 \colon S \lr C_2/G$.
Take any point over $\bar{y} \in C_2/G$ and let $F$ denote the
schematic fibre of $\alpha_2$ over $\bar{y}$. Then
\begin{itemize}
\item[$(i)$] The reduced structure of $F$ is the union of an
irreducible curve $Y$, called the central component of $F$,
 and either none or at least two mutually disjoint HJ-strings, each
 meeting $Y$ at one point, and each being contracted by
 $\lambda$ to a singular point of $T$.
 These strings are in one-to-one
 correspondence with the branch points of $C_1 \lr C_1/H$, where $H
 \subseteq G$ is the stabilizer of $y$.
\item[$(ii)$] The intersection of a string with $Y$ is transversal,
and it takes place at only one of the end components of the string.
\item[$(iii)$] $Y$ is isomorphic to $C_1/H$, and has multiplicity
equal to $|H|$ in $F$.
\end{itemize}
An analogous statement holds if one considers the natural projection
$\alpha_1 \colon S \lr C_1/G$.
\end{theorem}
In the sequel we denote by $\mathscr{H}(F)$ the set of the
$HJ$-strings contained in $F$ and we say that $F$ is a
\emph{reducible fibre} if $\mathscr{H}(F) \neq \emptyset$. Theorem
\ref{Serrano} therefore implies
\begin{remark} \label{card H}
For every reducible fibre $F$, the cardinality of $\mathscr{H}(F)$
is at least two.
\end{remark}
For a proof of the following result, see \cite[p. 509-510]{Bar99},
\cite{Fre71}, \cite{MiPol08}.

\begin{proposition} \label{invariants-S}
Let $\lambda \colon S \lr T=(C_1 \times C_2)/G$ be a standard
isotrivial fibration. Then the invariants of $S$ are given by
\begin{itemize}
\item[$(i)$] $K_S^2 =\frac{8(g(C_1)-1)(g(C_2)-1)}{|G|} + \sum \limits_{x \in \emph{Sing}\;
T}h_x$;
\item[$(ii)$] $e(S)=\frac{4(g(C_1)-1)(g(C_2)-1)}{|G|}
+\sum \limits_{x \in \emph{Sing}\; T}e_x$;
\item[$(iii)$] $q(S)=g(C_1/G)+g(C_2/G)$.
\end{itemize}
\end{proposition}

\begin{corollary} \label{chi-S}
Let $\lambda \colon S \lr T=(C_1 \times C_2)/G$ be a standard
isotrivial fibration. Then
\begin{equation} \label{eq:chi-S}
K_S^2= 8 \chi(\mO_S)- \frac{1}{3} \sum_{x \in \emph{Sing }T}B_x.
\end{equation}
\end{corollary}
\begin{proof}
Proposition \ref{invariants-S} yields $K_S^2=2e(S)-\sum_{x \in
\textrm{Sing }T}(2e_x-h_x)$. By Noether's formula we have $K_S^2=12
\chi(\mO_S) -e(S)$, so \eqref{eq:chi-S} follows.
\end{proof}
Let us consider now the minimal resolution of a cyclic quotient
singularity $x \in T$. If $Y_1$ and $Y_2$ are the strict
transforms of $C_1$ and $C_2$, by Theorem \ref{Serrano} we obtain
the situation illustrated in Figure \ref{figure-res}.
\begin{figure}[H]
\begin{center}
\includegraphics*[totalheight=6 cm]{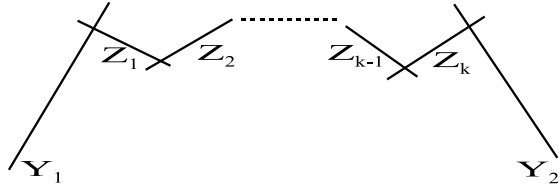}
\end{center}
\caption{Resolution of a cyclic quotient singularity $x \in T$}
\label{figure-res}
\end{figure}
The curves $Y_1$ and $Y_2$ are the central components of two
reducible fibres $F_1$ and $F_2$ of $\alpha_2 \colon S \lr C_2/G$
and $\alpha_1 \colon S \lr C_1/G$, respectively. Then there exist
$\lambda_1, \ldots, \lambda_k, \, \mu_1, \ldots, \mu_k \in
\mathbb{N}$ such that
\begin{equation} \label{eq:fibres}
\begin{split}
F_1&=\rho_1 Y_1+\sum_{i=1}^k \lambda_i Z_i + \Gamma_1, \\
F_2&=\rho_2 Y_2+ \sum_{i=1}^k \mu_i Z_i + \Gamma_2,
\end{split}
\end{equation}
where the supports of both divisors $\Gamma_1$ and $\Gamma_2$ are union of $HJ$-strings
disjoint from the $Z_i$; moreover if $x$ is of type
$\frac{1}{n}(1,q)$, then
 $n$ divides both $\rho_1$ and $\rho_2$. Now we have
\begin{equation} \label{eq:system-lambda}
\left\{ \begin{array}{l} 0=F_1 Z_k=-\lambda_k b_k
+\lambda_{k-1} \\
0 = F_1 Z_{k-1} = \lambda_k - b_{k-1} \lambda_{k-1} +
\lambda_{k-2} \\
\cdots \\
0= F_1Z_2 = \lambda_3 - b_2 \lambda_2 + \lambda_1 \\
0= F_1 Z_1 = \lambda_2 - b_1 \lambda_1 + \rho_1,
\end{array} \right.
\end{equation}
which gives
\begin{equation*} 
\left\{ \begin{array}{l} \lambda_{k-1} / \lambda_k=b_k \\
\lambda_{k-2}/ \lambda_{k-1}=[b_{k-1},\, b_k] \\
\cdots \\
 \lambda_1 / \lambda_2=[b_2, \, b_3, \ldots, b_k] \\
 \rho_1 / \lambda_1=[b_1, \, b_2, \ldots, b_k].
\end{array} \right.
\end{equation*}
In particular
\begin{equation} \label{lambda}
\lambda_1=\frac{\rho_1}{[b_1, b_2, \, \ldots b_k]}= \frac{\rho_1
q}{n}.
\end{equation}
Analogously, we have
\begin{equation*} 
\left\{ \begin{array}{l}  \mu_2 / \mu_1=b_1 \\
\mu_3 / \mu_2=[b_2,\, b_1] \\
\cdots \\
\mu_k / \mu_{k-1}=[b_{k-1}, \, b_{k-2}, \ldots, b_1] \\
 \rho_2 / \mu_k =[b_k, \, b_{k-1}, \ldots, b_1],
\end{array} \right.
\end{equation*}
hence
\begin{equation} \label{mu}
\mu_k=\frac{\rho_2}{[b_k, \, b_{k-1}, \ldots, b_1]}=\frac{\rho_2
q'}{n}.
\end{equation}

\begin{definition}
We say that a reducible fibre $F_1$ of $\alpha_2 \colon S \lr C_2/G$
 is of type $\big( \frac{q_1}{n_1}, \ldots, \frac{q_r}{n_r} \big)$ if
it contains exactly $r$ $HJ$-strings $\mE_1, \ldots, \mE_r$, where
each $\mE_i$ is of type $\frac{1}{n_i}(1, q_i)$. The same definition
holds for a reducible fibre $F_2$ of $\alpha_1 \colon S \lr C_1/G$.
\end{definition}

\begin{proposition} \label{selfintersection}
Let $F_1$ be of type $\big( \frac{q_1}{n_1},  \ldots,
\frac{q_r}{n_r} \big)$ and let $Y_1$ be its central component. Then
\begin{equation} \label{selfint-1}
(Y_1)^2 = - \sum_{i=1}^r \frac{q_i}{n_i}.
\end{equation}
Analogously, if $F_2$ is of type $(\frac{q_1}{n_1},  \ldots,
\frac{q_r}{n_r})$ then
\begin{equation} \label{selfint-2}
(Y_2)^2 = - \sum_{i=1}^r \frac{q_i'}{n_i}.
\end{equation}
\end{proposition}
\begin{proof}
If $F_1$ is of type $\type$, set $\rho_1=\textrm{l.c.m.}(n_1,
\ldots, n_r)$ and
\begin{equation*}
\mE_i:= \bigcup_{j=1}^{k_i}Z_{j, \, i} \quad \quad i=1, \ldots,r.
\end{equation*}
Then we can write
\begin{equation*}
 F_1=\rho_1Y_1 + \sum_{i=1}^r \sum
_{j=1}^{k_i}\lambda_{j,\,i}\,Z_{j,\,i}.
\end{equation*}
By using (\ref{lambda}), we have
\begin{equation*}
0=F_1Y=\rho_1(Y_1)^2+\sum_{i=1}^r \lambda_{1,i}=\rho_1(Y_1)^2+\rho_1
 \sum_{i=1}^r \frac{q_i}{n_i}
\end{equation*}
and this proves (\ref{selfint-1}). Analogously, one can use
(\ref{mu}) in order to prove (\ref{selfint-2}).
\end{proof}

\begin{corollary} \label{integers}
Assume $\emph{Sing}(T)= \frac{1}{n_1}(1, \, q_1)+ \cdots
+\frac{1}{n_r}(1, \, q_r)$. Then both
\begin{equation*}
\sum_{i=1}^r \frac{q_i}{n_i} \quad \emph{and} \quad \sum_{i=1}^r
\frac{q_i'}{n_i}
\end{equation*}
are integers.
\end{corollary}
\begin{corollary} \label{nodes}
Assume that  $T$ contains exactly $r$ ordinary double points as singularities. Then
$r$ is even.
\end{corollary}

\section{The non-minimal case} \label{sec:non-minimal}

Let $\lambda \colon S \lr T:=(C_1 \times C_2)/G$ be a standard
isotrivial fibration. If $g(C_1/G) \geq 1$ and $g(C_2/G) \geq 1$ then $S$ is
necessarily a minimal model. If instead $g(C_1/G)=0$, it may happen
that the central component of some reducible fibre $F_1$ of
$\alpha_2 \colon S \lr C_2/G$ is a $(-1)$-curve. Analogously, if
$g(C_2/G)=0$ it may happen that the central component of some
reducible fibre $F_2$ of $\alpha_1 \colon S \lr C_1/G$ is a
$(-1)$-curve.

\begin{definition}
We say that a reducible fibre $F_1$ of $\alpha_2 \colon S \lr C_2/G$
is a $(-1)$-\emph{fibre} if its central component $Y_1$
 is a $(-1)$-curve. If $g(C_1)=\mg$, we will also say that $F$
is a $(-1)$-\emph{fibre in genus} $\mg$. The same definitions hold
for a reducible fibre $F_2$ of $\alpha_1 \colon S \lr C_1/G$.
\end{definition}

\begin{proposition} \label{(-1)-fibre}
Assume that $F_1$ is a reducible fibre of $\alpha_2 \colon S \lr
C_2/G$, of type $\type$. Set $\rho:=\emph{l.c.m.}(n_1, \ldots, n_r)$.
Then $F_1$ is a $(-1)$-fibre if and only if
\begin{equation*}
\sum_{i=1}^r \frac{q_i}{n_i}=1 \quad \emph{and}  \quad 2g(C_1)-2 =
\rho \St.
\end{equation*}
Assume that $F_2$ is a reducible fibre of $\alpha_1 \colon S \lr
C_1/G$, of type $\type$. Then $F_2$ is a $(-1)$-fibre if and only if
\begin{equation*}
\sum_{i=1}^r \frac{q_i'}{n_i}=1 \quad \emph{and} \quad 2g(C_2)-2=
\rho \St.
\end{equation*}
\begin{proof}
Let us consider first $F_1$. By Proposition
\ref{selfintersection} and Theorem \ref{Serrano} the two conditions
are equivalent to $(Y_1)^2=-1$ and $g(Y_1)=0$, respectively.
If we consider $F_2$ the proof is analogous.
\end{proof}
\end{proposition}

The following result provide a method to construct non-minimal standard isotrivial
fibrations with arbitrarily many $(-1)$-fibres.

\begin{theorem} \label{existence}
Let $\mathcal{S}:=\big\{ \frac{q_1}{n_1}, \ldots, \frac{q_r}{n_r}
\big\} $ be a finite set of rational numbers, with $(n_i, q_i)=1$, such that
$\sum_{i=1}^r \frac{q_i}{n_i}=1$. Set $n:=\emph{l.c.m.}(n_1, \ldots,
n_r)$. Then for any $\mathfrak{q} \geq 0$ there exists a standard isotrivial
fibration $\lambda \colon S \lr T:=(C_1 \times C_2)/G$ such that
the following holds.
\begin{itemize}
\item[$(i)$] $\emph{Sing}(T)=n \times \frac{1}{n_1}(1,q_1)+ \cdots +
 n \times \frac{1}{n_r}(1, q_r);$
\item[$(ii)$] the singular fibres of the natural projection
 $\,\alpha_2 \colon S \lr C_2/G$ are exactly $\;n\;$ $(-1)$-fibres, all of type
 $\type;$ \\
\item[$(iii)$] $q(S)=\mathfrak{q}$.
\end{itemize}
\end{theorem}
\begin{proof}
For all $i \in \{i, \ldots, r \}$ set $t_i:=q_in/n_i$. Set moreover
$G:= \langle \xi \, | \, \xi^n=1 \rangle \cong \mathbb{Z}_n$. Since
$(n, \, t_i)= n/n_i$, the element $\xi^{t_i}$ has order
$n/(n, \,t_i)=n_i$ in $G$. It follows that $G$ is both $(0
\; | \; n_1, \ldots, n_r)$ and $(\mathfrak{q} \; | \; n^n)$-generated, with
generating vectors given by
\begin{equation*}
\begin{split}
\mathcal{V}_1=\{g_1, \ldots, g_r \}&:=\{\xi^{t_1}, \ldots, \xi^{t_r}
\} \quad
\textrm{and}\\
\mathcal{V}_2=\{\ell_1, \ldots, \ell_n; \;\; h_1, \ldots,
h_{2\mathfrak{q}}\}&:=\{ \underbrace{\xi, \ldots, \xi}_{n \textrm{ times}};
\;\; \underbrace{\xi, \ldots, \xi}_{2\mathfrak{q} \textrm{ times}} \},
\end{split}
\end{equation*}
respectively. Therefore by Proposition \ref{riemann ext} we obtain
two $G$-covers
\begin{equation*}
 C_1 \lr C_1/G \cong \mathbb{P}^1, \quad  C_2 \lr C_2/G,
\end{equation*}
where $g(C_2/G)=\mathfrak{q}$. By using Proposition \ref{fixed-points} we see
that
\begin{itemize}
\item for all $i \in \{1, \ldots, r \}$, there are $n/n_i$ fixed
points on $C_1$ with stabilizer $\langle \xi^{t_i} \rangle \cong
\mZ_{n_i}$; if $P_i$ is the set of these fixed points, we have
\begin{equation*}
| \textrm{Fix}_{C_1, \,q}(\xi^{t_i}) \cap P_i|= \left\{
\begin{array}{ll}
 n/n_i & \textrm{if $q=1$} \\
0 & \textrm{otherwise};
\end{array} \right.
\end{equation*}
\item there are $n$ fixed points on $C _2$, whose stabilizer is
the whole $G$; for all $i \in \{1, \ldots, r \}$ we have
\begin{equation*}
| \textrm{Fix}_{C_2, \,q}(\xi^{t_i})|= \left\{
\begin{array}{ll}
n & \textrm{if $q=q_i$} \\
0 & \textrm{otherwise}.
\end{array} \right.
\end{equation*}
\end{itemize}
It follows that the standard isotrivial fibration $\lambda \colon S
\lr T =(C_1 \times C_2)/G$ has all the desired properties.
\end{proof}

In the sequel we will focus our attention on the natural projection $\alpha_2
\colon  S \lr C_2/G$; this involves no loss of generality and
similar results hold if one considers instead the projection
$\alpha_1 \colon S \lr C_1/G$. For abbreviation, we simply write
``$(-1)$-fibre" instead of
 ``$(-1)$-fibre of $\alpha_2 \colon S \lr C_2/G$".

\begin{corollary} \label{class -1 fibres}
The classification of $(-1)$-fibres in genus $\mg$ is equivalent to the
classification of pairs $(G, \mathcal{S})$, where $G$ is a finite group
and $\mathcal{S}:=\big\{ \frac{q_1}{n_1}, \ldots, \frac{q_r}{n_r}
\big\} $ is a set of rational numbers, with $(n_i, q_i)=1$ for all $i$, such that
\begin{itemize}
\item[$(i)$] $G$ acts in genus $\mg$ with rational quotient and
signature $(0 \, | \, n_1, \ldots, n_r)$;
\item[$(ii)$] $\sum_{i=1}^r \frac{q_i}{n_i}=1$.
\end{itemize}
\end{corollary}
\begin{proof}
Immediate by Proposition \ref{(-1)-fibre} and Theorem
\ref{existence}.
\end{proof}

\begin{corollary} \label{g=0}
The following are equivalent:
\begin{itemize}
\item[$(i)$] $F$ is a $(-1)$-fibre in genus $\mg=0$;
\item[$(ii)$] $F$ is a reducible fibre in genus $\mg=0$;
\item[$(iii)$] $F$ is a reducible fibre of type
 $\big(\frac{q}{n}, \, \frac{n-q}{n} \big)$ whose central component is rational.
\end{itemize}
\end{corollary}
\begin{proof} $(i)\Rightarrow (ii)$. Obvious. \\
$(ii) \Rightarrow (iii)$. Assume $g(C_1)=0$. For all $n \geq 2$,
the cyclic group $\mathbb{Z}_n$ acts on
$\mathbb{P}^1 $, and the only possible signature is $(0 \;| \;n,n)$
(\cite[p. 9]{Bre00}). Therefore every reducible fibre of $\alpha_2
\colon S \lr C_2/G$ is of
type $\big( \frac{q_1}{n}, \frac{q_2}{n} \big)$ for some positive integers $n$, $q_1$, $q_2$.
On the other hand we have seen that
$\frac{q_1}{n}+\frac{q_2}{n}$ must be integer, so
 $F$ is of type $\big( \frac{q}{n},
\frac{n-q}{n} \big)$. Finally, the central component of $F$ is rational since it
is a quotient of $C_1$ (Theorem \ref{Serrano}).\\
$(iii) \Rightarrow (i)$. This follows from Proposition \ref{(-1)-fibre}.
\end{proof}

Corollary \ref{g=0} shows that there are infinitely many types of
$(-1)$-fibres in genus $\mg=0$. On the other hand, for all genera $\mg \geq
1$ there are only finitely many types, since there are only
finitely many cyclic groups of automorphisms; the cases where
$\mg=1, \, 2, \, 3$ are described in detail in Appendix A.

\begin{example} \label{Example 1-a}
Let $n \geq 2$ be any positive integer and take
$\mathcal{S} = \big\{ \frac{1}{n}, \, \frac{n-1}{n} \big\}$, $\mathfrak{q}=1$.
Using the construction given in Theorem \ref{existence}, we obtain a
standard isotrivial fibration $\lambda \colon S \lr T=(C_1 \times
C_2)/G$ with
\begin{equation*}
g(C_1)=0, \quad 2g(C_2)-2=n^2-n, \quad \textrm{Sing}(T)=n
\times \frac{1}{n}(1,\,1)+ n \times \frac{1}{n}(1, \, n-1).
\end{equation*}
For all $n$, $S$ is a ruled surface whose invariants are
$p_g(S)=0$, $q(S)=1$, $K_S^2=-n^2$. Hence every minimal model $\widehat{S}$
 of $S$ satisfies $K_{\widehat{S}}^2=0$.
\end{example}

\begin{example} \label{Example-1}
Take $\mathcal{S}=\bigg\{\underbrace{\frac{1}{n}, \ldots,
\frac{1}{n} }_{\textrm{$n$ times}}\bigg\}$ and $\mathfrak{q}=1$.
We obtain a
standard isotrivial fibration with
\begin{equation*}
2g(C_1)-2=n^2-3n, \quad 2g(C_2)-2=n^2-n, \quad \textrm{Sing}(T)=n^2
\times \frac{1}{n}(1,\,1).
\end{equation*}
Thus Proposition \ref{invariants-S} yields
\begin{equation*}
\begin{split}
K_S^2 &=n^3-4n^2+2n, \quad e(S)=n^3-2n^2+2n, \\
\chi(\mO_S)&=\frac{n(n-1)(n-2)}{6}, \quad q(S)=1.
\end{split}
\end{equation*}
For $n=2$, $S$ is a ruled surface. Now we assume $n \geq 3$. Since
$\mathfrak{q} >0$, the minimal model $\widehat{S}$ of $S$ is
obtained by contracting $n$ disjoint $(-1)$-curves. Hence its
invariant are
\begin{equation*}
K_{\widehat{S}}^2 =n(n-1)(n-3), \quad e(\widehat{S})=n(n-1)^2.
\end{equation*}
For $n =3$ we obtain an elliptic surface with
$\textrm{kod}(\widehat{S})=1$ and
$p_g(\widehat{S})=q(\widehat{S})=1$, whose elliptic fibration
$\alpha_2 \colon S \lr C_2/G$ contains exactly three singular elements, all
 of type $IV(\widetilde{A}_2)$ according to Kodaira classification
(\cite[Chapter V]{BPV}); for $n \geq 4$ we have
a surface of general type. Taking $\mathfrak{q}> 1$ leads to similar results:
for $n=3$ the surface $\widehat{S}$ is elliptic and satisfies
$p_g(\widehat{S})=q(\widehat{S})=\mathfrak{q}$, whereas for $n \geq 4$ it is
of general type.
\end{example}

\begin{remark}
Under the assumptions of Theorem $\ref{existence}$, one may ask
whether there exists a standard isotrivial fibration such that
$\emph{Sing}(T)=\frac{1}{n_1}(1, \, q_1) + \cdots +\frac{1}{n_r}(1,
\, q_r)$. In general the answer is negative, in fact further
necessary conditions are
\begin{equation*}
\frac{1}{3} \sum_{i=1}^r B\bigg(\frac{q_i}{n_i} \bigg) \in
\mathbb{Z} \quad \emph{and} \quad \sum_{i=1}^r \frac{q'_i}{n_i} \in
\mathbb{Z},
\end{equation*}
see Corollaries $\ref{chi-S}$ and $\ref{integers}$. For example,
there are no standard isotrivial
 fibrations with $\emph{Sing}(T)= 3 \times \frac{1}{3}(1, \,1)$
 or with $\emph{Sing}(T)= 2 \times \frac{1}{5}(1, \,1)+ \frac{1}{5}(1,3)$.
 In some cases, however, the question above has an affirmative
answer. For instance, in \cite{MiPol08} there are examples of
standard isotrivial fibrations with $\emph{Sing}(T)=4 \times
\frac{1}{4}(1,1)$ and with $\emph{Sing}(T)= \frac{1}{7}(1,1)
+\frac{1}{7}(1,2)+ \frac{1}{7}(1,4)$.
\end{remark}

\section{The relatively minimal model}
\label{sec:relatively minimal}

\subsection{Contractible components} \label{sub:contractible-comp}
Let $\lambda \colon S \lr T=(C_1 \times C_2)/G$ be a standard
isotrivial fibration. If $F$ is any $(-1)$-fibre of $\alpha_2 \colon S \lr C_2/G$,
with $\mathscr{H}(F)=\{\mE_1, \cdots, \mE_r \}$, we consider the following procedure:
\begin{itemize}
\item[\emph{Step} $0$]: contract the central component $Y$ of $F$;
\item[\emph{Step} $1$]: make all possible contractions in the image of $\mE_1$;
\item[\emph{Step} $2$]: make all possible contractions in the image of
$\mE_2$;
\item[ ] $\cdots$
\item[\emph{Step} $r$]: make all possible contractions in the image of $\mE_r$;
\item[\emph{Step} $r$]$+1:$ go back to Step $1$ and repeat.
\end{itemize}
Applying this algorithm to all $(-1)$-fibres, we obtain a relative
minimal fibration $\hat{\alpha}_2 \colon \widehat{S} \lr C_2/G$.
If $g(C_1) \geq 1$ this is the unique relative minimal model of
$\alpha_2$ (\cite[Chapter III, Proposition 8.4]{BPV}); by abuse of
terminology, we will say that $\hat{\alpha}_2$ is \emph{the}
relative minimal model of $\alpha_2$ also when $g(C_1)=0$. If
$g(C_2/G) \geq 1$, then $\widehat{S}$ is obviously a minimal
surface. If $g(C_2/G)=0$ this is not true in general, as following
example illustrates.
\begin{example} \label{no-minimal}
The group $G=\textrm{PSL}_2(\mathbb{F}_7)$ has
order $168$ and it is
$(0 \, | \, 2,3,7)$-generated (\cite[p. 265-266]{JS87}). Then there exists a genus
$3$ curve $C$ and
a $G$-cover $C \lr \mathbb{P}^1$, branched in three
points with branching numbers $2$, $3$ and $7$, respectively. Set $C_1=C_2=C$ and
consider the standard isotrivial fibration $\lambda \colon S \lr T=(C_1 \times C_2)/G$;
standard computations as in \cite{MiPol08} show that
\begin{equation*}
\textrm{Sing}(T)=4 \times \frac{1}{2}(1, \, 1)+\frac{1}{3}(1, \, 1)+
\frac{1}{3}(1, \, 2)+ \frac{1}{7}(1, \, 1)+ \frac{1}{7}(1, \, 2)
+ \frac{1}{7}(1, \,4).
\end{equation*}
By using Proposition \ref{invariants-S} we obtain
\begin{equation*}
K_S^2=-6, \quad e(S)=18, \quad q(S)=0,
\end{equation*}
hence $\chi(\mO_S)=1$ and $p_g(S)=0$.
The natural projection $\alpha_2 \colon S \lr C/G \cong \mathbb{P}^1$
contains precisely three reducible fibres $F_2$, $F_3$, $F_7$ and moreover:
\begin{itemize}
\item $F_2$ is of type $\big(\frac{1}{2}, \, \frac{1}{2}, \,
\frac{1}{2}, \,\frac{1}{2} \big)$;
\item $F_3$ is of type $\big(\frac{1}{3}, \, \frac{2}{3} \big)$;
\item $F_7$ is of type $\big(\frac{1}{7}, \, \frac{2}{7}, \,
\frac{4}{7} \big)$.
\end{itemize}
\end{example}
Out of these, the unique $(-1)$-fibre
is $F_7$, in fact the central components of $F_2$ and $F_3$ are not rational
curves. The surface $\widehat{S}$ is therefore obtained by blowing down two curves
(see Example \ref{7} below), hence $K_{\widehat{S}}^2=-4$
and consequently  $\widehat{S}$ is \emph{not} a minimal surface. It is no difficult
to check that in this example $\textrm{kod}(S)=- \infty$. \\ \\
Now let $\lambda \colon S \lr T=(C_1 \times C_2)/G$ be a standard
isotrivial fibration and $\hat{\alpha}_2 \colon \widehat{S} \lr
C_2/G$ the relative minimal model of $\alpha_2$. Let $F$ be a
reducible fibre of $\alpha_2$ and let $\mE= \bigcup_{i=1}^k Z_i \in
\mathscr{H}(F)$ be a $HJ$-string contained in $F$. We say that an
irreducible component $Z_i \subset \mE$ is \emph{contractible} if it
is contracted by the natural map $\pi \colon S \lr \widehat{S}$. By
definition it follows that if  both $Z_i$ and $Z_j$ are
contractible, then $Z_l$ is also contractible for any $i \leq l \leq
j$. Now we define
\begin{equation*}
\cF:= \textrm{number of irreducible components of $F$ contracted by
$\pi$}.
\end{equation*}
Obviously, $\cF \geq 0$; moreover $\cF>0$ if and only if $F$ is a
$(-1)$-fibre, and $\cF=1$ if and only if $F$ is a $(-1)$-fibre and
none of its $HJ$-strings contains contractible components.
\begin{example} \label{7}
If $F$ is a $(-1)$-fibre of type $\big( \frac{1}{3}, \, \frac{1}{3},
\, \frac{1}{3} \big)$, then $\cF=1$. If $F$ is a $(-1)$-fibre of
type $\big( \frac{1}{7}, \, \frac{2}{7}, \, \frac{4}{7} \big)$, then
$\cF=2$.
\end{example}
For any $\tau \in C_2/G$, let $F_{\tau}$ and $(F_{\tau})_{\textrm{red}}$
be the fibre and the reduced fibre of $\alpha_2 \colon S \lr C_2/G$ over $\tau$,
respectively, and set
\begin{equation*}
\begin{split}
\textrm{Crit}(\alpha_2)&:= \{ \tau \in C_2/G \;\; | \;\; F_{\tau}
\textrm{ is singular} \};\\
\mathcal{R}(\alpha_2)&:=\{ \tau \in
\textrm{Crit}(\alpha_2) \;\; | \;\;
(F_{\tau})_{\textrm{red}} \textrm{ is smooth} \}; \\
\textrm{Crit}(\alpha_2)'&:=\textrm{Crit}(\alpha_2) \setminus
\mathcal{R}(\alpha_2)= \{\tau \in C_2/G \;\; | \;\; F_{\tau}
\textrm{ is reducible} \}.
\end{split}
\end{equation*}
Moreover, given any reducible fibre $F$, let us define
\begin{equation*}
\delta(F):=\frac{1}{3} \sum_{\mE \in \mathscr{H}(F)}B(\mE)  - \cF.
\end{equation*}
\begin{example}
If $F$ is of type $\bigg(\underbrace{\frac{1}{n}, \ldots,
\frac{1}{n} }_{\textrm{$n$ times}}\bigg)$, with $n \geq 3$, then
$\delta(F)= \frac{1}{3}(n^2-1)$ if $F$ is a $(-1)$-fibre and $\delta(F)=\frac{1}{3}(n^2+2)$
otherwise. If $n=2$ then $\delta(F)=0$ if $F$ is a $(-1)$-fibre and $\delta(F)=2$
otherwise.
\end{example}
The rational number $\delta(F)$ plays an important role in the sequel,
because of the following result.
\begin{proposition}
With the above notations we have
\begin{equation} \label{eq:chi-S-hat}
K_{\widehat{S}}^2 = 8 \chi (\mO_{\widehat{S}})- \sum_{\tau \in
\emph{Crit}(\alpha_2)' } \delta(F_\tau).
\end{equation}
\end{proposition}
\begin{proof}
Immediate by using \eqref{eq:chi-S} and the definition of
$\delta(F)$.
\end{proof}

\begin{remark}
S.L. Tan pointed out that one has the equality
\begin{equation*}
\delta(F)=\frac{1}{3}(2c_2(F)-c_1^2(F)),
\end{equation*}
where $c_1^2(F)$ and $c_2(F)$ are the invariants defined in
\cite{Tan96}.
\end{remark}

The behaviour of $\delta(F)$ when $F$ is not a $(-1)$-fibre is quite
simple.
\begin{lemma} \label{no (-1)-fibre}
Let $F$ be a reducible fibre which is \emph{not} a $(-1)$
fibre. Then $\delta(F) \geq 2$ and equality holds if and only if $F$
 is of type $\big( \frac{1}{2}, \, \frac{1}{2} \big)$.
\end{lemma}
\begin{proof}
Since $F$ is not a $(-1)$-fibre we have $\cF=0$; moreover
$\mathscr{H}(F)$ contains at least two $HJ$-strings (Remark
\ref{card H}), so Remark \ref{geq 3} yields
\begin{equation*}
\delta(F) = \frac{1}{3} \sum_{\mE \in \mathscr{H}(F)}B(\mE) \geq 2
\end{equation*}
and equality holds if and only if $\mathscr{H}(F)$ contains exactly
two $HJ$-strings, both of type  $\frac{1}{2}(1,1)$.
\end{proof}

Now we start the analysis of the case where $F$ is a $(-1)$-fibre.
If $\mE \in \mathscr{H}(F)$ is a $HJ$-string of type $\frac{1}{n}(1,
\,q)$, with $\frac{n}{q}=[b_1, \ldots, b_k]$, we define
$b_i(\mE):=b_i$ for all $1 \leq i \leq k$. In particular,
$-b_1(\mE)=-\lceil \frac{n}{q}\rceil$ equals the self-intersection
of the unique curve in $\mE$ which meets the central component $Y$
of $F$.

\begin{lemma} \label{cardinality}
Assume that $F$ is a $(-1)$-fibre of $\alpha_2 \colon S \lr C_2/G$
and set $\mathscr{H}(F)=\{\mE_1,
\ldots, \mE_r \}$. Then
\begin{itemize}
\item[$(i)$] $\cF=1$ if and only if $b_1(\mE_i) \geq 3$ for all
$i;$ \item[$(ii)$] the set $\{i \; | \;b_1(\mE_i)=2 \}$ has
cardinality at most two, and it has cardinality two if and only if
$F$ is of type $\big(\frac{1}{2}, \frac{1}{2} \big)$. If this
happens, then $S$ is ruled$;$ \item[$(iii)$] if $r=2$ and $F$ is
not of type $\big( \frac{1}{2}, \, \frac{1}{2} \big)$,  we may
assume $b_1(\mE_1)=2$ and $b_1(\mE_2) \geq 3$.
\end{itemize}
\end{lemma}
\begin{proof}
We have $\cF=1$ if and only if no further $(-1)$-curves arise in $F$
after contracting its central component; this is in
turn equivalent to say that $b_1(\mE_i) \geq 3$ for all $i$, so
our first claim is proven. \\
Now let us assume $b_1(\mE_1)=2$; hence $\frac{n_1}{q_1} \leq 2$,
that is $\frac{q_1}{n_1} \geq \frac{1}{2}$. Therefore by using
\eqref{selfint-1} we obtain $\sum_{i \geq 2} \frac{q_i}{n_i} = 1 - \frac{q_1}{n_1}
\leq
\frac{1}{2}$, which in turn implies  $b_1(\mE_i) =\lceil
\frac{n_i}{q_i} \rceil \geq 3$ for all $i \geq 2$, unless $F$ contains
exactly two strings $\mE_1$, $\mE_2$, both of type
$\frac{1}{2}(1,1)$. In this case, contracting the central
component we obtain
 two $(-1)$-curves intersecting transversally in a point; therefore
 by \cite[Proposition 4.6 p. 79]{BPV} it follows $\textrm{kod}(S)=-\infty$, that is
 $S$ is ruled. This proves $(ii)$. \\
Finally, assume $\mathscr{H}(F)= \{\mE_1, \, \mE_2 \}$. In this case $\mg=0$ and, by
Corollary \ref{g=0}, $F$ is of type $\big(\frac{q}{n}, \, \frac{n-q}{n} \big)$.
 We may assume $\frac{q}{n} >
\frac{1}{2}$; hence $\frac{n}{q} < 2$ and $b_1(\mE_1)=\lceil
\frac{n}{q}\rceil=2$. Now part $(ii)$ gives $b_1(\mE_2)
\geq 3$.
\end{proof}


\begin{proposition} \label{delta=0}
All $(-1)$-fibres in genus $0$ satisfy $\delta(F)=0$.
\end{proposition}
\begin{proof}
By Corollary \ref{g=0}, any $(-1)$-fibre $F$ in genus $0$ is of type
$\big(\frac{q}{n}, \, \frac{n-q}{n} \big)$. If
$\frac{q}{n}=\frac{n-q}{n}=\frac{1}{2}$ the result is clear.
Otherwise by Riemenschneider's duality \eqref{duality} it follows
\begin{equation*}
\cF= \ell \bigg(\frac{q}{n} \bigg)+ \ell \bigg(\frac{n-q}{n} \bigg)=
\sum_{i=1}^t (k_i+1) + \sum_{i=1}^{t-1}(l_i+1),
\end{equation*}
hence Proposition \ref{duality-2} implies $\delta(F)=0$.
\end{proof}

\begin{remark}
If $g(C_1)=0$ then Proposition \emph{\ref{delta=0}} and relation
\eqref{eq:chi-S-hat} imply $K_{\widehat{S}}^2 = 8 \chi
(\mO_{\widehat{S}})$, according to the fact that $\hat{\alpha}_2
\colon \widehat{S} \lr C_2/G$ is a relatively minimal rational
fibration.
\end{remark}


\begin{lemma} \label{no-contr-3}
Let $F$ be a $(-1)$-fibre in genus $\mg \geq 1$. Then
$\mathscr{H}(F)=\{\mE_1, \ldots, \mE_r \}$ with $r \geq 3$. Moreover
we may assume that $\mE_i$ contains no contractible components for
$i \geq 3$.
\end{lemma}
\begin{proof}
If $\mathscr{H}(F)=\{\mE_1, \, \mE_2 \}$ then $\mg=0$ by Corollary
\ref{g=0}. Then $\mathscr{H}(F)= \{ \mE_1, \ldots, \mE_r \}$ with $r
\geq 3$. Suppose $r=3$ and put
\begin{equation*}
\mE_1= \bigcup Z_i, \quad \mE_2= \bigcup W_j, \quad \mE_3=\bigcup
T_h.
\end{equation*}
If $\cF=1$ there is nothing to prove. Thus we can assume $\cF
>1$ and $b_1(\mE_1)=-(Z_1)^2=2$; by Lemma \ref{cardinality} we
have $b_1(E_2)\geq 3$ and $b_1(E_3)\geq 3$. Let us write
\begin{equation}
\mE_1 \colon [(2)^k, \, l+3, \ldots]=\frac{n_1}{q_1},  \quad k>0, \; l \geq 0;
\end{equation}
therefore we can contract the central component $Y$ of $F$ and the
images of $Z_1, \ldots, Z_k$, but not the image of $Z_{k+1}$. After
these contractions, the images of the curves $W_1$ and $T_1$ are
tangent at one point. If also $W_1$ can be contracted, then the
image of $T_1$ becomes \emph{singular}, hence $\mE_3$ contains no
contractible components. If $r \geq 4$ the argument is the same.
\end{proof}

\begin{proposition} \label{cF=1}
Let $F$ be a $(-1)$-fibre such that $\cF=1$. Then $\delta(F) \geq 2
+ \frac{2}{3}$ and equality holds if and only if $F$ is of type
$\big(\frac{1}{3}, \frac{1}{3}, \frac{1}{3} \big)$.
\end{proposition}
\begin{proof}
Since $\cF=1$ we have $\mathscr{H}(F)= \{ \mE_1, \ldots, \mE_r \}$,
with $r \geq 3$ and $b_1(\mE_i) \geq 3$ for all $i$ (Lemma \ref{cardinality}).
Thus Corollary
\ref{cor-B-increasing} implies
\begin{equation*}
\delta(F)  = \frac{1}{3} \sum_{i=1}^r B(\mE_i)-1 \geq \frac{1}{3}
\cdot 3 \cdot B \bigg(\frac{1}{3} \bigg)-1 = 2 + \frac{2}{3}
\end{equation*}
and equality holds if and only if $\mathscr{H}(F)= \big \{ \mE_1, \,
\mE_2, \, \mE_3 \big \}$ and all $\mE_i$ are of type
$\frac{1}{3}(1,1)$.
\end{proof}

\begin{proposition} \label{cases-genus-no-0}
Let $F$ be a $(-1)$-fibre in genus $\mg \geq 1$ such that $\cF \geq
2$. If $\mathscr{H}(F)= \{\mE_1, \ldots, \mE_r \}$, then $\mE_1$ and
$\mE_2$ belong to one of the following cases.
\begin{equation*}
\begin{split}
\mathbf{Case \; 1.} \; \; & \mE_1 \colon [(2)^{k_1}, \, ***]\\
& \mE_2 \colon [***] \\
 &
\end{split}
\end{equation*}
\begin{equation*}
\begin{split}
\mathbf{Case \; 2.} \; \; & \mE_1 \colon [(2)^{k_1}, \, ***]\\
& \mE_2 \colon [k_1+2, \, ***] \\
 &
\end{split}
\end{equation*}
\begin{equation*}
\begin{split}
\mathbf{Case \; 3.} \; \; & \mE_1 \colon [(2)^{k_1}, \; l_1+3, \;
(2)^{k_2}, \ldots, (2)^{k_{t-1}}, \; l_{t-1}+3,
\; (2)^{k_t}, \, ***]  \\
& \mE_2 \colon [k_1+2, \; (2)^{l_1}, \; k_2+3, \ldots,  k_{t-1}+3,
\; (2)^{l_{t-1}}, \;
 k_t+3, \, ***] \quad t \geq 1,\\
 &
\end{split}
\end{equation*}
\begin{equation*}
\begin{split}
\mathbf{Case \; 4.} \; \; & \mE_1 \colon [(2)^{k_1}, \; l_1+3, \;
(2)^{k_2}, \ldots, (2)^{k_{t-1}}, \; l_{t-1}+3,
\; (2)^{k_t}, \, ***]  \\
& \mE_2 \colon [k_1+2, \; (2)^{l_1}, \; k_2+3, \ldots,  k_{t-1}+3,
\;
(2)^{l_{t-1}}, \, ***] \quad t \geq 2, \\
 &
\end{split}
\end{equation*}
where $k_i, \,l_j \geq 0$ and $``***"$ denotes the non-contractible
part of the $HJ$-string.
\end{proposition}
\begin{proof}
By Lemma \ref{no-contr-3} we may assume that all contractible components
of $F$, different from the central component $Y$, belong to $\mE_1 \cup \mE_2$.
Moreover, since $\cF \geq 2$ and $\mg \geq 1$, we can suppose $b_1(\mE_1)
=2$ and $b_1(\mE_2) \geq 3$ (Lemma \ref{cardinality}). Set
\begin{equation*}
\mE_1= \bigcup Z_i, \quad \mE_2= \bigcup W_j
\end{equation*}
and
\begin{equation*}
\begin{split}
\mE_1 \colon & [(2)^{k_1}, \, l_1+3, \, (2)^{k_2}, \, l_2+3,
\ldots] \quad
k_i, \, l_j \geq 0, \, k_1>0, \\
\mE_2 \colon & [u_1+3, \, (2)^{v_1}, \, u_2+3, \, (2)^{v_2},
\ldots] \quad u_i, \, v_j \geq 0.
\end{split}
\end{equation*}
Now we start the contraction process described in Subsection
\ref{sub:contractible-comp}; since $\mg \geq 1$, it never gives rise
to rational curves with self-intersection equal to $0$. First, we
can contract the central component $Y$ and the images of the curves
$Z_1, \ldots, Z_{k_1}$, but not the image of $Z_{k_1+1}$; then
either we stop or the image of $W_1$ has self-intersection $(-1)$,
that forces $u_1=k_1-1$. In this case we can contract the images of
$W_1, \ldots, W_{v_1+1}$, but not the image of $W_{v_1+2}$; then
either we stop or the image of $Z_{k_1+1}$ has self-intersection
$(-1)$, which gives $v_1=l_1$. In the same way we obtain
\begin{equation*}
u_i=k_i \; \; \textrm{and} \; \; v_i=l_i \quad \textrm{for all} \;i
\geq 2.
\end{equation*}
Repeated application of this argument yields either one of Cases $1,
\ldots, 4$ described in the statement or one of Cases $3'$, $4'$
below.
\begin{equation*}
\begin{split}
\mathbf{Case \; 3'.} \; \; & \mE_1 \colon [(2)^{k_1}, \; l_1+3, \;
(2)^{k_2}, \ldots, (2)^{k_t}, \; l_t+3, \, ***] \\
& \mE_2 \colon [k_1+2, \; (2)^{l_1}, \; k_2+3, \ldots,  k_t+3, \;
(2)^{l_t}, \, ***] \\
 &
\end{split}
\end{equation*}
\begin{equation*}
\begin{split}
\mathbf{Case \; 4'.} \; \; & \mE_1 \colon [(2)^{k_1}, \; l_1+3, \;
(2)^{k_2}, \ldots, (2)^{k_t}, \, ***] \\
& \mE_2 \colon [k_1+2, \; (2)^{l_1}, \; k_2+3, \ldots,  k_t+3, \;
(2)^{l_t}, \, ***]. \\
 &
\end{split}
\end{equation*}
Finally we observe that Case $3'$ (resp. Case $4'$) is obtained
by putting $k_1=0$ in Case $3$ (resp. in Case $4$) and interchanging
$\mE_1$ and $\mE_2$. This completes the proof.
\end{proof}

\begin{proposition} \label{delta-magg-0}
Let $F$ be a $(-1)$-fibre in genus $\mg \geq 1$. Then $\delta(F) >2$
with exactly the following two exceptions:
\begin{itemize}
\item[($i$)] $F$ is of type  $\big(
\frac{1}{2}, \frac{1}{3}, \frac{1}{6} \big);$ in this case
$\mg=1$ and $\delta(F)=1+\frac{1}{3}.$ \\
\item[($ii$)] $F$ is of type  $\big(
\frac{1}{2}, \frac{1}{4}, \frac{1}{4} \big);$ in this case $\mg=1$
and $\delta(F)=2.$
\end{itemize}
\end{proposition}
\begin{proof}
Set $\mathscr{H}(F)= \{\mE_1, \ldots, \mE_r \}$, where each $\mE_i$
is of type $\frac{1}{n_i}(1,\, q_i)$; by Lemma \ref{no-contr-3}
 we have $r \geq 3$.
Since we dealt with the case $\cF=1$ in Proposition \ref{cF=1}, we
may assume $\cF \geq 2$. Moreover by Lemma \ref{no-contr-3} we can
suppose that $\mE_i$ contains no contractible components for $i \geq
3$. We will discuss Cases $1, \ldots, 4$ of Proposition
\ref{cases-genus-no-0} separately.
\begin{equation*}
\begin{split}
\mathbf{Case \; 1.} \; \; & \mE_1 \colon [(2)^{k_1}, \, ***] \\
& \mE_2 \colon [***].
\end{split}
\end{equation*}
In this case
\begin{equation} \label{A1}
\cF=\ell \bigg(\frac{k_1}{k_1+1} \bigg)+1=k_1+1.
\end{equation}
By Propositions \ref{decreasing} and \ref{B-increasing} it follows
\begin{equation*}
\frac{q_1}{n_1} \geq \frac{k_1}{k_1+1} \; \; \textrm{and} \; \; B
\bigg( \frac{q_1}{n_1} \bigg) \geq B \bigg( \frac{k_1}{k_1+1}
\bigg)=2k_1 + \frac{2k_1}{k_1+1}.
\end{equation*}
Moreover
\begin{equation*}
\sum_{i=2}^r\frac{q_i}{n_i}= 1 - \frac{q_1}{n_1} \leq
1-\frac{k_1}{k_1+1}= \frac{1}{k_1+1}.
\end{equation*}
Then we may assume
\begin{equation*}
\frac{q_2}{n_2} \leq \frac{1}{(r-1)(k_1+1)} \leq \frac{1}{2(k_1+1)}
\end{equation*}
hence $b_1(\mE_2) = \lceil \frac{n_2}{q_2} \rceil \geq 2(k_1+1)$;
moreover $b_1(\mE_3) \geq k_1+3$ since $\mE_3$ contains no
contractible components. Thus Corollary \ref{cor-B-increasing} implies
\begin{equation*}
B\bigg(\frac{q_2}{n_2}\bigg) \geq 2(k_1+1)+ \frac{1}{k_1+1}, \quad
B\bigg(\frac{q_3}{n_3} \bigg) \geq k_1+3 + \frac{2}{k_1+3}.
\end{equation*}
Then
\begin{equation} \label{B1}
\begin{split}
\delta(F)& \geq \frac{1}{3}B\big(\mE_1 \big)+ \frac{1}{3}B\big(\mE_2
\big) +\frac{1}{3}B\big(\mE_3 \big)-\cF \\
& \geq \frac{1}{3} \bigg(2k_1+4+ \frac{k_1-1}{(k_1+1)(k_1+3)} \bigg)
\geq 2
\end{split}
\end{equation}
and equality holds if and only if $F$ is of type $\big( \frac{1}{2},
\, \frac{1}{4}, \, \frac{1}{4} \big)$.
\begin{equation*}
\begin{split}
 & \\
\mathbf{Case \; 2.} \; \; & \mE_1 \colon [(2)^{k_1}, ***]\\
& \mE_2 \colon [k_1+2, ***] \quad k_1 \geq 1. \\
\end{split}
\end{equation*}
In this case
\begin{equation} \label{A2}
 \cF=\ell \bigg(\frac{k_1}{k_1+1} \bigg)+\ell
\bigg(\frac{1}{k_1+2} \bigg)+1=k_1+2.
\end{equation}
By Proposition \ref{decreasing} it follows
\begin{equation*}
\frac{q_1}{n_1} \geq \frac{k_1}{k_1+1}, \quad \frac{q_2}{n_2} \geq
\frac{1}{k_1+2}
\end{equation*}
and Proposition \ref{B-increasing} implies
\begin{equation*}
\begin{split}
B \bigg(\frac{q_1}{n_1} \bigg) & \geq B \bigg(\frac{k_1}{k_1+1}
\bigg)= 2k_1+\frac{2k_1}{k_1+1}, \\ B \bigg(\frac{q_2}{n_2} \bigg) &
\geq B \bigg(\frac{1}{k_1+2} \bigg)= k_1+2+\frac{2}{k_1+2}.
\end{split}
\end{equation*}
Moreover
\begin{equation*}
\sum_{i=3}^r \frac{q_i}{n_i}= 1 - \frac{q_1}{n_1}- \frac{q_2}{n_2} \leq 1- \frac{k_1}{k_1+1}
-\frac{1}{k_1+2}= \frac{1}{(k_1+1)(k_1+2)}.
\end{equation*}
Then we may assume
\begin{equation*}
\frac{q_3}{n_3} \leq \frac{1}{(r-2)(k_1+1)(k_1+2)} \leq
\frac{1}{(k_1+1)(k_1+2)}
\end{equation*}
hence $b_1(\mE_3) = \lceil \frac{n_3}{q_3} \rceil \geq
(k_1+1)(k_1+2)$. Thus Corollary \ref{cor-B-increasing} yields
\begin{equation*}
B\bigg( \frac{q_3}{n_3} \bigg) \geq
(k_1+1)(k_1+2)+\frac{2}{(k_1+1)(k_1+2)}.
\end{equation*}
Therefore we obtain
\begin{equation} \label{B2}
\begin{split}
\delta(F)& \geq \frac{1}{3}B\big(\mE_1 \big)+ \frac{1}{3}B\big(\mE_2
\big) + \frac{1}{3}B\big(\mE_3 \big)-\cF \\ & \geq
 \frac{1}{3}k_1(k_1+3) \geq 1+ \frac{1}{3}
\end{split}
\end{equation}
and equality holds if and only if $F$ is
of type $\big( \frac{1}{2}, \, \frac{1}{3}, \, \frac{1}{6}
\big)$.
\begin{equation*}
\begin{split}
 & \\
\mathbf{Case \; 3.} \; \; & \mE_1 \colon [(2)^{k_1}, \; l_1+3, \;
(2)^{k_2}, \ldots, (2)^{k_{t-1}}, \; l_{t-1}+3,
\; (2)^{k_t}, ***] \\
& \mE_2 \colon [k_1+2, \; (2)^{l_1}, \; k_2+3, \ldots,  k_{t-1}+3,
\; (2)^{l_{t-1}}, \;
 k_t+3, ***], \quad t \geq 2.
\end{split}
\end{equation*}
Let $n, \,q$ be coprime integers such that
\begin{equation*}
[(2)^{k_1}, \; l_1+3, \; (2)^{k_2}, \ldots, (2)^{k_{t-1}}, \;
l_{t-1}+3, \; (2)^{k_t}]= \frac{n}{n-q'}
\end{equation*}
and let $a$ be such that $qq'=1 + a n$. Then Proposition \ref{prop-fraz-cont-1} yields
\begin{equation*}
[ k_1+2, \; (2)^{l_1}, \; k_2+3, \ldots, k_{t-1}+3, \;
(2)^{l_{t-1}}, \; k_t+3] = \frac{n+q}{a+q'}.
\end{equation*}
Notice that if $\frac{n}{n-q'}=2$ then
$\frac{n+q}{a+q'}=3$, and by interchanging $\mE_1$ and $\mE_2$ we are in
Case $2$; hence we may assume $n \geq 3$. We have
\begin{equation} \label{A3}
\cF=\ell \bigg(\frac{n-q'}{n} \bigg)+\ell \bigg(\frac{a+q'}{n+q}
\bigg)+1= \sum_{i=1}^t(k_i+1) + \sum_{i=1}^{t-1}(l_i+1)+1.
\end{equation}
By Proposition \ref{decreasing} it follows
\begin{equation*}
\frac{q_1}{n_1} \geq \frac{n-q'}{n}, \quad \frac{q_2}{n_2} \geq
\frac{a + q'}{n+q}
\end{equation*}
and Proposition \ref{B-increasing} gives
\begin{equation} \label{case3-1}
B \bigg( \frac{q_1}{n_1} \bigg) \geq B \bigg( \frac{n-q'}{n} \bigg),
\quad B \bigg( \frac{q_2}{n_2} \bigg) \geq B \bigg( \frac{a +
q'}{n+q} \bigg).
\end{equation}
Moreover
\begin{equation*}
\sum_{i=3}^r \frac{q_i}{n_i} \leq 1-\frac{n-q'}{n} -\frac{a +
q'}{n+q} =\frac{1}{n(n+q)}.
\end{equation*}
Then we may assume
\begin{equation*}
\frac{q_3}{n_3} \leq \frac{1}{(r-2)n(n+q)} \leq \frac{1}{n(n+q)}
\end{equation*}
hence $b_1(\mE_3) = \lceil \frac{n_3}{q_3} \rceil \geq n(n+q)$. By
Corollary \ref{cor-B-increasing} this implies
\begin{equation} \label{case3-2}
B \bigg( \frac{q_3}{n_3} \bigg) \geq n(n+q)+ \frac{2}{n(n+q)}.
\end{equation}
Estimates \eqref{case3-1} and \eqref{case3-2} together with
\eqref{fraz-cont-21} now yield
\begin{equation} \label{B3}
\begin{split}
\delta(F)& \geq \frac{1}{3}B\big(\mE_1 \big)+ \frac{1}{3}B\big(\mE_2
\big) + \frac{1}{3}B\big(\mE_3 \big)- \cF \\
 & \geq \frac{1}{3} \bigg(1 -
 \frac{1+q^2}{n(n+q)}+n(n+q)+\frac{2}{n(n+q)} \bigg)-1.
\end{split}
\end{equation}
Since $n \geq 3$ we obtain $\delta(F) > 3$.
\begin{equation*}
\begin{split}
 & \\
\mathbf{Case \; 4.} \; \; & \mE_1 \colon [(2)^{k_1}, \; l_1+3, \;
(2)^{k_2}, \ldots, (2)^{k_{t-1}}, \; l_{t-1}+3,
\; (2)^{k_t}, ***] \\
& \mE_2 \colon [k_1+2, \; (2)^{l_1}, \; k_2+3, \ldots,  k_{t-1}+3,
\; (2)^{l_{t-1}}, ***], \quad t \geq 2.
\end{split}
\end{equation*}
Let $n, \,q$ be coprime integers such that
\begin{equation*}
[(2)^{k_1}, \; l_1+3, \; (2)^{k_2}, \ldots, (2)^{k_{t-1}}, \;
l_{t-1}+3, \; (2)^{k_t}]= \frac{n}{n-q'}
\end{equation*}
and let $a$ be such that $qq'=1 + a n$. Then Proposition \ref{prop-fraz-cont-1} yields
\begin{equation*}
[ k_1+2, \; (2)^{l_1}, \; k_2+3, \ldots, k_{t-1}+3, \;
(2)^{l_{t-1}}] = \frac{q}{a}.
\end{equation*}
Notice that $q \geq 2$. If $n=3$, $q=2$ we
obtain $\frac{n}{n-q'}=3$, $\frac{q}{a}=2$, so by interchanging $\mE_1$
and $\mE_2$ we are in Case $2$; analogously if $n=4$, $q=3$. Hence
we can suppose $n \geq 5$. We have
\begin{equation} \label{A4}
\cF=\ell \bigg(\frac{n-q'}{n} \bigg)+\ell \bigg(\frac{a}{q}
\bigg)+1=\sum_{i=1}^t(k_i+1) + \sum_{i=1}^{t-1}(l_i+1).
\end{equation}
By Proposition \ref{decreasing} it follows
\begin{equation*}
\frac{q_1}{n_1} \geq \frac{n-q'}{n}, \quad \frac{q_2}{n_2} \geq
\frac{a}{q}
\end{equation*}
and Proposition \ref{B-increasing} gives
\begin{equation} \label{case4-1}
B \bigg( \frac{q_1}{n_1} \bigg) \geq B \bigg( \frac{n-q'}{n} \bigg),
\quad B \bigg( \frac{q_2}{n_2} \bigg) \geq B \bigg( \frac{a}{q}
\bigg).
\end{equation}
Moreover
\begin{equation*}
\sum_{i=3}^r \frac{q_i}{n_i}= 1 - \frac{q_1}{n_1}-\frac{q_2}{n_2} \leq 1-\frac{n-q'}{n} -\frac{a}{q}
 =\frac{1}{nq}.
\end{equation*}
Then we may assume
\begin{equation*}
\frac{q_3}{n_3} \leq \frac{1}{(r-2)nq} \leq \frac{1}{nq}
\end{equation*}
hence $b_1(\mE_3) = \lceil \frac{n_3}{q_3} \rceil \geq nq$. By
Corollary \ref{cor-B-increasing} this implies
\begin{equation} \label{case4-2}
B \bigg( \frac{q_3}{n_3} \bigg) \geq nq+\frac{2}{nq}.
\end{equation}
Estimates \eqref{case4-1} and \eqref{case4-2} together with
\eqref{fraz-cont-22} now yield
\begin{equation} \label{B4}
\begin{split}
\delta(F)& \geq \frac{1}{3}B\big(\mE_1 \big)+ \frac{1}{3}B\big(\mE_2
\big) + \frac{1}{3}B\big(\mE_3 \big)- \cF \\
&  \geq \frac{1}{3}\bigg(n-\frac{1}{n} \bigg) \bigg(q-\frac{1}{q}
\bigg).
\end{split}
\end{equation}
Since $n \geq 5$ and $q \geq 2$  it follows $\delta(F) \geq 2+
\frac{2}{5}$.
\end{proof}
Summarizing Lemma \ref{no (-1)-fibre}, Proposition
\ref{delta=0} and Proposition \ref{delta-magg-0} we obtain
\begin{corollary} \label{fibres}
Let $F$ be a reducible fibre of $\alpha_2 \colon S \lr  C_2/G$.
Then $\delta(F) \geq 0$ and moreover the following holds. \\
$\bullet$ If $g(C_1)=0$ then $F$ is a $(-1)$-fibre  and
$\delta(F)=0$. Conversely,
if $\delta(F)=0$ then $F$ is a $(-1)$-fibre and $g(C_1)=0$. \\
$\bullet$ If $g(C_1) \geq 1$ then $\delta(F) >2$, with precisely three
exceptions:
\begin{itemize}
\item[($i$)] $g(C_1)=1$ and $F$ is a $(-1)$-fibre of type
$\big(\frac{1}{2}, \, \frac{1}{3}, \, \frac{1}{6} \big)$. In this
case $\delta(F)= 1+\frac{1}{3};$
\item[($ii$)] $g(C_1)=1$ and $F$ is a $(-1)$-fibre of type
$\big(\frac{1}{2}, \, \frac{1}{4}, \, \frac{1}{4} \big)$. In this
case $\delta(F)=2;$
\item[($iii$)] $F$ is of type $\big(\frac{1}{2}, \, \frac{1}{2} \big)$
but it is \emph{not} a $(-1)$-fibre. In this case $\delta(F)=2$.
\end{itemize}
In particular, if $S$ is of general type then the only possible
exception is $(iii)$.
\end{corollary}
Notice that in case $(iii)$ the central component $Y$ of $F$
satisfies $Y^2=-1$ but it is \emph{not} a rational curve.

\begin{proposition} \label{K2-min-8-chi }
Let $\lambda \colon S \lr T=(C_1 \times C_2)/G$ be a standard
isotrivial fibration and let $\hat{\alpha}_2 \colon \widehat{S} \lr
C_2/G$ be the relatively minimal model of $\alpha_2 \colon S \lr
C_2/G$. Then
\begin{equation*}
K_{\widehat{S}}^2 \leq 8 \chi(\mathcal{O}_{\widehat{S}})
\end{equation*}
 and equality holds if and only if either $S$ is a quasi-bundle or
 $g(C_1)=0$. Otherwise we have
\begin{equation*}
K_{\widehat{S}}^2 \leq 8 \chi(\mathcal{O}_{\widehat{S}})-2
\end{equation*}
and equality holds if and only if $\alpha_2$ contains exactly one
reducible fibre $F$, which is of type $\big(\frac{1}{2}, \,
\frac{1}{2})$ and which is \emph{not} a $(-1)$-fibre $($in
particular this implies $\widehat{S}=S)$.
\end{proposition}
\begin{proof}
By using formula \eqref{eq:chi-S-hat} and Corollary \ref{fibres} we
obtain $K_{\widehat{S}}^2 \leq 8 \chi(\mathcal{O}_{\widehat{S}})$,
and equality holds if and only if either
\begin{itemize}
\item[($i$)] $\textrm{Crit}(\alpha_2)'= \emptyset$, that is $S$ is
 a quasi-bundle, or
\item[($ii$)] $g(C_1)=0$.
\end{itemize}
Otherwise, since both $K_{\widehat{S}}^2$ and
$\chi(\mO_{\widehat{S}})$ are integers, we have $K_{\widehat{S}}^2
\leq 8 \chi(\mathcal{O}_{\widehat{S}})-2$, and equality holds if and
only if  $\alpha_2$ contains exactly one reducible fibre $F$ and
either
\begin{itemize}
\item[($i'$)] $g(C_1)=1$ and $F$
is a $(-1)$ fibre of type $\big(\frac{1}{2}, \, \frac{1}{4},
\frac{1}{4} \big)$, or
\item[($ii'$)] $F$ is of type
$\big(\frac{1}{2}, \, \frac{1}{2} \big)$, but it is not a
$(-1)$-fibre.
\end{itemize}
Assume now that case $(i')$ occurs. Therefore, by using
Proposition \ref{invariants-S} we would obtain $K_S^2=-2$ and
$e(S)=5$, contradicting Noether's formula. Therefore the only
possibility is $(ii')$.
\end{proof}

\begin{proposition} \label{improve-serrano}
Assume that $q(S) \geq 1$ and that $S$ is neither ruled nor a quasi-bundle.
Then, up to interchanging $C_1$ and $C_2$, the surface
 $\widehat{S}$ is the minimal model of $S$ and we have
\begin{equation} \label{bound-iso}
 K_{\widehat{S}}^2 \leq 8 \chi(\mathcal{O}_{\widehat{S}})-2.
\end{equation}
Equality holds if and
only if $\emph{Sing}(T)= 2 \times \frac{1}{2}(1,1)$, and in this case
$S=\widehat{S}$ is a minimal surface of general type.
\end{proposition}
\begin{proof}
Consider the relatively minimal fibration $\hat{\alpha}_2 \colon \widehat{S} \lr C_2/G$.
 Since $q(S) \geq 1$, up to interchanging $C_1$ and $C_2$
  we can suppose $g(C_2/G) \geq 1$, hence $\widehat{S}$ is
the minimal model of $S$. We are also assuming
 that $S$ is not ruled, so $g(C_1) \geq 1$ and Proposition
\ref{K2-min-8-chi } gives $K_{\widehat{S}}^2 \leq 8
\chi(\mathcal{O}_{\widehat{S}})-2$. Equality occurs if and only if
$\alpha_2$ contains exactly one reducible fibre, which is of type $\big(\frac{1}{2}, \,
\frac{1}{2})$ and which is not a $(-1)$-fibre; this implies
$S=\widehat{S}$, hence $S$ is minimal and consequently $K_S$ is nef.
Therefore relation $K_S^2= 8 \chi(\mO_S)-2$ yields
$K_S^2 \geq 6$, that is $S$ is of general type.
\end{proof}

\begin{corollary} \label{nongen}
Let $S$ be a standard isotrivial fibration, with $\emph{kod}(S)=0$ or
$1$ and $\chi(\mO_S)=0$. Then $S$ is a quasi-bundle.
\end{corollary}
\begin{proof}
Since $\chi(\mO_S)=0$ we obtain $q(S) \geq 1$, hence
 $\widehat{S}$ is the minimal model of $S$. Now
$\textrm{kod}(S)=0$ or $1$  yields
$0=K_{\widehat{S}}^2=8\chi(\mO_{\widehat{S}})$,
so Proposition \ref{K2-min-8-chi }
implies that $S$ is a quasi-bundle.
\end{proof}

\begin{remark} \label{applies}
If $\emph{kod}(S)=0$, then Corollary $\ref{nongen}$ applies when $S$ is
  either abelian or bielliptic. If instead $\emph{kod}(S)=1$, it applies
  when $S$ is any
properly elliptic surface with $\chi(\mO_S)=0$ $($examples of such
surfaces are described in \cite{Se92}$)$. \\
Finally, observe that there exist $($non-minimal$)$ properly elliptic surfaces
  with $\chi(\mO_S)=1$ that are  standard isotrivial fibrations
but not quasi-bundles, see Example $\ref{Example-1}$.
This show that the assumption $\chi(\mO_S)=0$ in Corollary $\ref{nongen}$
cannot be dropped.
\end{remark}

Under the further assumption that $K_{\widehat{S}}$ is ample, we can improve
 inequality \eqref{bound-iso} as follows.

\begin{proposition} \label{improve-serrano-ample}
Assume that $q(S) \geq 1$,
$S$ is not a quasi-bundle and $K_{\widehat{S}}$ is ample.
Then, up to interchanging $C_1$ and $C_2$, the surface
 $\widehat{S}$ is the minimal model of $S$ and we have
\begin{equation} \label{bound-iso-ample}
 K_{\widehat{S}}^2 \leq 8 \chi(\mathcal{O}_{\widehat{S}})-5.
\end{equation}
\end{proposition}
\begin{proof}
By Proposition \ref{improve-serrano} we must show that,
 if $K_{\widehat{S}}$ is ample, the two cases
$K_{\widehat{S}}^2 = 8 \chi(\mathcal{O}_{\widehat{S}})-3$ and
$K_{\widehat{S}}^2 = 8 \chi(\mathcal{O}_{\widehat{S}})-4$ do not
occur. This will be consequence of Lemmas \ref{claim-3} and
\ref{claim-4} below.
\begin{lemma} \label{claim-3}
If $K_{\widehat{S}}$ is ample, then
$K_{\widehat{S}}^2 = 8 \chi(\mathcal{O}_{\widehat{S}})-3$ does not occur.
\end{lemma}
By contradiction, assume that this case occurs. Since $\widehat{S}$
is of general type, by formula \eqref{eq:chi-S-hat} and Corollary \ref{fibres}
it follows
that $\alpha_2 \colon S \lr C_2/G$ contains exactly one reducible fibre $F$,
 which satisfies $\delta(F)=3$. Assuming that $F$ is of type
 $\type$, there are two subcases. \\ \\
\emph{Subcase} $(1)$. $F$ is not a $(-1)$-fibre. This implies
$S=\widehat{S}$ and $\sum_{i=1}^r B \big(\frac{q_i}{n_i} \big)=9$.
Since $\sum_{i=1}^r \frac{q_i}{n_i} \in \mathbb{Z}$, by looking at
the table in Appendix B we see that the only possibility for the
type of $F$ is $\big(\frac{1}{3}, \, \frac{2}{3} \big)$, see also
\cite[Proposition 4.1]{MiPol08}, and this contradicts the
ampleness of the canonical bundle. Hence $(1)$ does not occur.
\\ \\
\emph{Subcase} $(2)$. $F$ is a $(-1)$-fibre. By using estimates
\eqref{B1}, \eqref{B2}, \eqref{B3}, \eqref{B4}, we can check that the
only possibilities for the type of $F$ are
 $\big(\frac{2}{3}, \, \frac{1}{6}, \, \frac{1}{6} \big)$ and
 $\big(\frac{1}{2}, \, \frac{1}{8}, \, \frac{3}{8} \big)$,
 see also Appendix A.
But in the latter case $K_{\widehat{S}}$
would not be ample, hence $F$ is necessarily of type
$\big(\frac{2}{3}, \, \frac{1}{6}, \, \frac{1}{6} \big)$. Therefore $g(C_1)=2$.
Moreover, since $F$ is a $(-1)$-fibre, we have $g(C_1/G)=0$; setting $\mg':=g(C_2/G)$,
it follows that $G$ is both $(0 \, | \, \mathbf{m})$-generated and
$(\mg' \, | \, \mathbf{n})$-generated, where $\mathbf{m}:=(m_1, \ldots, m_r)$ and
$\mathbf{n}:=(n_1, \ldots, n_s)$; we will denote by
\begin{equation} \label{gen-vec}
\mathcal{V}:=\{g_1, \ldots, g_r \} \quad \textrm{and} \quad
\mathcal{W}:=\{\ell_1, \ldots, \ell_s; \; h_1, \ldots,
h_{2\mathfrak{g}'} \}
\end{equation}
the corresponding generating vectors, see Section \ref{sec:prel}. The group $G$ acts in genus
$2$ with rational quotient; moreover, since
\begin{equation} \label{SingT-1}
\textrm{Sing}(T)=\frac{1}{3}(1, \,2)+ 2 \times \frac{1}{6}(1, \, 1),
\end{equation}
at least one of the $m_i$ must be divisible by $6$. Looking at
\cite[p. 252]{Br90} and \cite[Appendix A]{Pol07}, we see that there are
at most two possibilities: \\
\begin{itemize}
\item[$(2a)$] $G=\mathbb{Z}_2 \times \mathbb{Z}_6, \quad
\mathbf{m}=(2,\, 6^2)$; \item[$(2b)$] $G=\mathbb{Z}_2 \ltimes
((\mZ_2)^2 \times \mZ_3)=G(24, \, 8), \quad \mathbf{m}= (2,\, 4,
\, 6)$,
\end{itemize}
$ $ \\
where $G=G(24, \, 8)$ means that $G$ has the label number $8$ in the
\verb|GAP4| list of groups of order $24$, see \cite{Pol07}.
Let us analyze $(2a)$ and $(2b)$ separately. \\ \\
\emph{Assume $(2a)$ occurs}. Set $G=\langle x, \, y \, |
\,x^2=y^6=[x, \, y]=1 \rangle$. Up to automorphisms, we may suppose
\begin{equation*}
\begin{split}
g_1 &=x, \quad g_2=xy^{-1}, \quad g_3=y, \\
\ell_1 & =y.
\end{split}
\end{equation*}
Set $\mathscr{S}:=\langle g_1 \rangle \cup \langle g_2 \rangle \cup
\langle g_3 \rangle$. Since $G$ is abelian, $s \geq 2$ (Remark
\ref{abelian-gen}); moreover there is just one reducible fibre, so
we must have
\begin{equation*}
\langle \ell_2 \rangle  \cup \cdots \cup \langle \ell_s \rangle
\subseteq G \setminus \mathscr{S}= \{xy^2, \, xy^4 \}.
\end{equation*}
But this is impossible, since $(xy^2)^2=y^4 \in \mathscr{S}$ and
$(xy^4)^2=y^2 \in \mathscr{S}$. Therefore $(2a)$
does not occur. \\ \\
\emph{Assume $(2b)$ occurs}. The presentation of $G=G(24, \,8)$ is
\begin{equation*}
\begin{split}
G=\langle x,\, y,\, z,\, w \, |& \, x^2=y^2=z^2=w^3=1, \\
& [y,\,z]=[y, \, w]=[z, \, w]=1, \\
& xyx^{-1}=y, \, xzx^{-1}=zy, \, xwx^{-1}=w^{-1} \rangle.
\end{split}
\end{equation*}
It is no difficult to check that this group contains exactly one
conjugacy class of elements of order $3$, namely
$\textrm{Cl}(w)=\{w, \, w^{-1}\}$.
In particular every element of order $3$ is conjugate to its
inverse, hence Corollary \ref{cor:fix} implies that if $T$
contains some singularity of type $\frac{1}{3}(1, \,2)$, it must
also contain some singularity of type $\frac{1}{3}(1, \, 1)$. But
this
contradicts \eqref{SingT-1}, hence $(2b)$ must be excluded too. \\
\\This completes the proof of Lemma \ref{claim-3}.

\begin{lemma} \label{claim-4}
If $K_{\widehat{S}}$ is ample, then
$K_{\widehat{S}}^2 = 8 \chi(\mathcal{O}_{\widehat{S}})-4$ does not occur.
\end{lemma}
Again, assume by contradiction that this case occurs. As in the proof of
Lemma \ref{claim-3}, we see that $\alpha_2 \colon S \lr C_2/G$ contains just
one reducible fibre, which must be a $(-1)$-fibre with $\delta(F)=4$.
By using estimates \eqref{B1}, \eqref{B2},
\eqref{B3}, \eqref{B4}, we see that the only possibilities for the type of $F$ are
 $\big(\frac{3}{4}, \, \frac{1}{8}, \, \frac{1}{8} \big)$ and
 $\big(\frac{1}{2}, \, \frac{1}{12}, \, \frac{5}{12} \big)$.
One immediately checks that in the latter case $K_{\widehat{S}}$
would not be ample, hence $F$ is necessarily of type
$\big(\frac{3}{4}, \, \frac{1}{8}, \, \frac{1}{8} \big)$.
Therefore $g(C_1)=3$. Moreover, since $F$ is a $(-1)$-fibre we
have $g(C_1/G)=0$; setting $\mg':=g(C_2/G)$, it follows that $G$
is both $(0 \, | \, \mathbf{m})$-generated and $(\mg' \, | \,
\mathbf{n})$-generated, where $\mathbf{m}:=(m_1, \ldots, m_r)$ and
$\mathbf{n}:=(n_1, \ldots, n_s)$; we will denote the corresponding
generating vectors as in \eqref{gen-vec}. The group $G$ acts in
genus $3$ with rational quotient; moreover, since
\begin{equation} \label{SingT-2}
\textrm{Sing}(T)=\frac{1}{4}(1, \,3)+ 2 \times \frac{1}{8}(1, \, 1),
\end{equation}
at least one of the $m_i$ must be divisible by $8$. Looking at
\cite[p. 252]{Br90} and \cite[Appendix A]{Pol07}, we see that there
are at most five possibilities:
\begin{itemize}
\item[$(a)$] $G=\mZ_2 \times \mZ_8, \quad \mathbf{m}=(2, \, 8^2)$,
\item[$(b)$] $G=D_{2, \, 8, \, 5}=G(16, \, 6), \quad \mathbf{m}=(2, \, 8^2)$,
\item[$(c)$] $G=\mZ_2 \ltimes (\mZ_2 \times \mZ_8)=G(32, \, 9), \quad \mathbf{m}=(2,\,4,\,8)$,
\item[$(d)$] $G=\mZ_2 \ltimes D_{2,8,5}=G(32, \, 11), \quad \mathbf{m}=(2,\,4,\,8)$,
\item[$(e)$] $G=\mathcal{S}_3 \ltimes (\mZ_4)^2=G(96, \, 64), \quad \mathbf{m}=(2, \, 3, \, 8)$.
\end{itemize}
We first rule out Case $(a)$. Set
\begin{equation*}
G= \langle x,\, y, \, | \, x^2=y^8=[x,\,y]=1 \rangle.
\end{equation*}
Up to automorphisms, we may assume
\begin{equation*}
\begin{split}
g_1& = x, \quad g_2=xy^{-1}, \quad g_3=y, \\
\ell_1&=y.
\end{split}
\end{equation*}
Set $\mathscr{S}:=\langle g_1 \rangle \cup \langle g_2 \rangle \cup
\langle g_3 \rangle$. Since $G$ is abelian, $s \geq 2$. Moreover
there is just one reducible fibre, so we must have
\begin{equation*}
\langle \ell_2 \rangle  \cup \cdots \cup \langle \ell_s \rangle
\subseteq G \setminus \mathscr{S}= \{xy^2, \, xy^4, \, xy^6 \}.
\end{equation*}
But $(xy^2)^2=(xy^6)^2=y^4 \in
\mathscr{S}$, so we obtain $\ell_2= \ldots =\ell_s=xy^4$. On the other hand,
\begin{equation*}
1=\ell_1 \ell_2 \cdots \ell_s \cdot
\Pi_{i=1}^{\mathfrak{g}'} [h_i,h_{i+\mathfrak{g}'}]= \ell_1 \ell_2 \cdots \ell_s,
\end{equation*}
so $y = \ell_1 \in \langle xy^4 \rangle$ which is a contradiction.
Hence $(a)$ must be excluded. \\ \\
Now we rule out Cases $(b), \ldots, (e)$. Notice that
\eqref{SingT-2} implies that the group $G$ must satisfy the
following condition:
\begin{itemize}
\item[$(*)$] there exists an element $g \in G$ such that $|g|=8$ and
$g$ is \emph{not} conjugate to $g^3$, $g^5$, $g^7$.
\end{itemize}
By using \verb|GAP4| (or by means of tedious hand-made
computations) we can easily check that in Cases $(b)$, $(d)$ and
$(e)$ every $g \in G$ with $|g|=8$ is conjugate to $g^5$, so condition $(\ast)$
is not satisfied. Therefore we are only left
to exclude $(c)$.
In Case $(c)$ the presentation of $G=G(32, \, 9)$ is
\begin{equation*}
G=\langle x ,\, y, \, z, \, | \,
x^2=y^2=z^8=1, \, [x,y]=[y,z]=1, \, xzx^{-1}=yz^3 \rangle.
\end{equation*}
By simple \verb|GAP4| scripts one checks that
the automorphism group $\textrm{Aut}(G)$ has order $64$, and that $G$
admits precisely $64$
generating vectors $\mathcal{V}=\{g_1, \, g_2, \, g_3 \}$
of type $(0 \,| \, 2, \, 4,\, 8)$,
 which form a unique orbit for the action of $\textrm{Aut}(G)$. Hence,
up to automorphisms, we may assume that $\mathcal{V}$ is as follows:
\begin{equation*}
g_1=x, \quad g_2=xz^{-1}, \quad g_3=z.
\end{equation*}
Set $\mg'=g(C_2/G)$ and
let $\mathcal{W}:=\{\ell_1, \ldots, \ell_s; \; h_1, \ldots,
h_{2\mathfrak{g}'} \}$ be the generating vector of type
$(\mg' \, | \, n_1, \ldots, n_s)$ inducing the covering $C_2 \lr C_2/G$.
The group $G$ contains no elements of order greater than $8$,
so by \eqref{SingT-2} we may assume $\ell_1=z$, and since
$z \notin [G,\,G]= \langle yz^2 \rangle$, we have $s \geq 2$.
Put
\begin{equation*}
\mathscr{S}:= \bigcup _{\sigma \in G} \bigcup_{i=1}^3 \langle \sigma g_i
\sigma^{-1} \rangle;
\end{equation*}
since $\alpha_2 \colon S \lr C_2/G$ contains
exactly one reducible fibre, we obtain
\begin{equation*}
\langle \ell_2 \rangle  \cup \cdots \cup \langle \ell_s \rangle
\subseteq G \setminus \mathscr{S}= \{yz^2, \, xz^2, \, xyz^2x, \, zxz, \,
z^2x, \, y, \, xy\} \subset \langle x, \, y, \, z^2 \rangle.
\end{equation*}
In particular this implies
\begin{equation} \label{exclude-1}
\ell_2 \ell_3 \ldots \ell_s \in \langle x, \, y, \, z^2 \rangle.
\end{equation}
On the other hand
\begin{equation} \label{exclude-2}
\ell_1\ell_2 \ldots \ell_s = \bigg(\Pi_{i=1}^{\mathfrak{g}'}
[h_i,h_{i+\mathfrak{g}'}]\bigg)^{-1} \in [G, \, G]=\langle yz^2 \rangle \subset \langle
x, \, y, \, z^2 \rangle,
\end{equation}
hence \eqref{exclude-1} and \eqref{exclude-2} together
imply $z=\ell_1 \in \langle
x, \, y, \, z^2 \rangle$, a contradiction. Therefore Case $(c)$ does not occur, and
this shows Lemma \ref{claim-4}. \\ \\
The proof of Proposition \ref{improve-serrano-ample} is now complete.
\end{proof}

In \cite{Se96} Serrano showed that any isotrivial fibred surface $X$
satisfies $K_X^2 \leq 8\chi(\mO_X)$. Moreover, S. L. Tan proved in
\cite{Tan96} that equality holds if and only if $X$ is either ruled
or isomorphic to a quasi-bundle. By using Propositions
\ref{improve-serrano} and \ref{improve-serrano-ample}, we are led to
the following strengthening of Serrano's and Tan's results.

\begin{theorem} \label{arbitrary-iso}
Let $\varphi \colon X \lr C$ be any relatively minimal
isotrivial fibration, with $X$
non ruled and $g(C) \geq 1$. If
 $X$ is not isomorphic to a quasi-bundle, we have
\begin{equation} \label{eq:gen-type-1}
K_X^2 \leq 8 \chi(\mO_X)-2
\end{equation}
and if equality holds then $X$ is a minimal surface of general
type whose canonical model
has precisely two ordinary double points as singularities. \\
Moreover, under the further assumption that $K_X$ is ample, we
have
\begin{equation} \label{eq:gen-type-1-ample}
K_X^2 \leq 8 \chi(\mO_X)-5.
\end{equation}
Finally, both inequalities \eqref{eq:gen-type-1} and
\eqref{eq:gen-type-1-ample} are sharp.
\end{theorem}
\begin{proof}
By Remark \ref{no-loss-if-standard} there exist a standard
isotrivial fibration $\lambda \colon S \lr T=(C_1 \times C_2)/G$ and a
birational map $T \dashrightarrow X$ such that the diagram
\begin{equation} \label{diagram-iso}
\xymatrix{ S \ar[d]_{\lambda}  \ar[rd]^{\psi} \\
T \ar@{-->}[r]  \ar[d]_{\sigma_2} &
 X \ar[d]^{\varphi} \\
 C_2/G \ar[r]^{\cong} & C}
\end{equation}
commutes. Since $\varphi$ is relatively minimal and $g(C) \geq 1$, the
surface $X$ is a minimal model. As $X$ is not ruled $K_X$ is nef, so the
rational map $\psi \colon S \lr X$ is actually a morphism, which
induces an isomorphism $\hat{\psi} \colon \widehat{S} \lr X$. Thus
Propositions \ref{improve-serrano} and \ref{improve-serrano-ample}
 imply inequalities \eqref{eq:gen-type-1} and \eqref{eq:gen-type-1-ample}.
Finally, both these inequalities are sharp, in fact:
\begin{itemize}
\item there exist examples of relatively minimal isotrivial fibrations $X \lr C$
 with $g(C)=1$, $p_g(X)=q(X)=1$ and $K_X^2=6$, see \cite[Section 7.1]{Pol07};
\item there exist examples of relatively minimal isotrivial
fibrations with $g(C)=1$, $p_g(X)=q(X)=1$, $K_S^2=3$ and $K_S$
ample, see \cite[Section 5.5]{MiPol08}. The fibres have genus $3$
and there is a unique singular fibre, composed of four $(-3)$ curves
intersecting in one single point.
\end{itemize}
This concludes the proof of Theorem \ref{arbitrary-iso}.
\end{proof}

\begin{remark}
If $K_X$ is not ample, then both cases $K_X^2=8 \chi(\mO_X)-3$ and
$K_X^2=8 \chi(\mO_X)-4$  actually occur. For instance, there are
examples of relatively minimal isotrivial fibrations with $g(C)=1$,
$p_g(X)=q(X)=1$ and $K_X^2=5,4$, see \cite{MiPol08} and
\cite{Pol07}.
\end{remark}

We end this section with an open problem.
\begin{problem}
What happens if one drops the assumptions $q(S) \geq 1$ in
Proposition $\ref{improve-serrano}$ and $g(C) \geq 1$ in Theorem
$\ref{arbitrary-iso}?$
\end{problem}


\section*{Appendix A. The classification of $(-1)$-fibres for
low values of $\mg$} For low values of $\mg$ there exists a complete
classification of cyclic groups acting in genus $\mg$ with rational
quotient; by Corollary \ref{class -1 fibres} this provides in turn a
complete classification of the corresponding $(-1)$-fibres. Since
Corollary \ref{g=0} settles the case $\mg=0$, we may assume $\mg
\geq 1$. If $F$ is any $(-1)$-fibre of $\alpha_2 \colon S \lr
C_2/G$, we denote by $F_{\textrm{min}}:=\pi(F)$ the image of $F$ in
the relative minimal model $\widehat{S}$.

\subsection{The case $\mg=1$}
\begin{proposition} \label{g=1}
There are precisely three types of $(-1)$-fibres $F$ in genus $\mg=1$.
The type of $F$, the values of $\cF$ and $\delta(F)$ and the type of
$F_{\emph{min}}$ in the Kodaira
classification of elliptic singular fibres are as in the table
below.
\begin{equation*}
\begin{tabular}{c|c|c|c}
\hline
$\emph{Type of }F$ & $\cF$ & $\delta(F)$ & $\emph{Type of }F_{\emph{min}}$ \\
 \hline
$\big( \frac{1}{3}, \, \frac{1}{3}, \, \frac{1}{3} \big)$ & $1$
& $2+\frac{2}{3}$ & $IV(\widetilde{A}_2)$ \\
$\big( \frac{1}{2}, \, \frac{1}{4}, \, \frac{1}{4} \big)$ & $2$ &
$2$ &
$III(\widetilde{A}_1)$ \\
$\big( \frac{1}{2}, \, \frac{1}{3}, \, \frac{1}{6} \big)$ & $3$ &
$1+\frac{1}{3}$ & $II$ \\ \hline
\end{tabular}
\end{equation*}
\end{proposition}
\begin{proof}
The cyclic groups $G$ acting in genus $1$ with rational quotient and
the corresponding signatures are as follows (\cite[p. 9]{Bre00}):
\begin{itemize}
\item[$(i)$] $G=\mathbb{Z}_2, \quad (0 \;| \; 2^4);$
\item[$(ii)$] $G=\mathbb{Z}_3, \quad (0 \, | \, 3^3);$
\item[$(iii)$] $G=\mathbb{Z}_4, \quad (0 \, | \, 2, 4^2);$
\item[$(iv)$] $G=\mathbb{Z}_6, \quad (0 \, | \, 2,3,6)$.
\end{itemize}
In case $(i)$ we cannot have a $(-1)$-fibre. \\
In case $(ii)$ a $(-1)$-fibre $F$ is necessarily of type $\big(
\frac{1}{3}, \, \frac{1}{3}, \, \frac{1}{3} \big)$; $F_{\textrm{min}}$
 is obtained by contracting only the central component, hence
 $\cF=1$ and $\delta(F)=B(\frac{1}{3})-1=2+\frac{2}{3}$. \\
In case $(iii)$ a $(-1)$-fibre is necessarily of type $\big(
\frac{1}{2}, \, \frac{1}{4}, \, \frac{1}{4} \big)$; $F_{\textrm{min}}$ is obtained by performing two blow-downs, hence
$\cF=2$ and $\delta(F)=\frac{1}{3}(B(\frac{1}{2})+ 2 B(\frac{1}{4}))-2=2$. \\
In case $(iv)$ a $(-1)$-fibre is necessarily of type $\big(
\frac{1}{2}, \, \frac{1}{3}, \, \frac{1}{6} \big)$; $F_{\textrm{min}}$
is  obtained by performing three blow-downs, hence $\cF=3$
and  $\delta(F)=\frac{1}{3}(B(\frac{1}{2})+ B(\frac{1}{3}) +
B(\frac{1}{6}))-3=1+\frac{1}{3}$. \\
In each case the blow-down process and the type of
$F_{\textrm{min}}$ are illustrated in Figure \ref{figure-minimal}.
This completes the proof.
\begin{figure}[H]
\begin{center}
\includegraphics*[totalheight=15 cm]{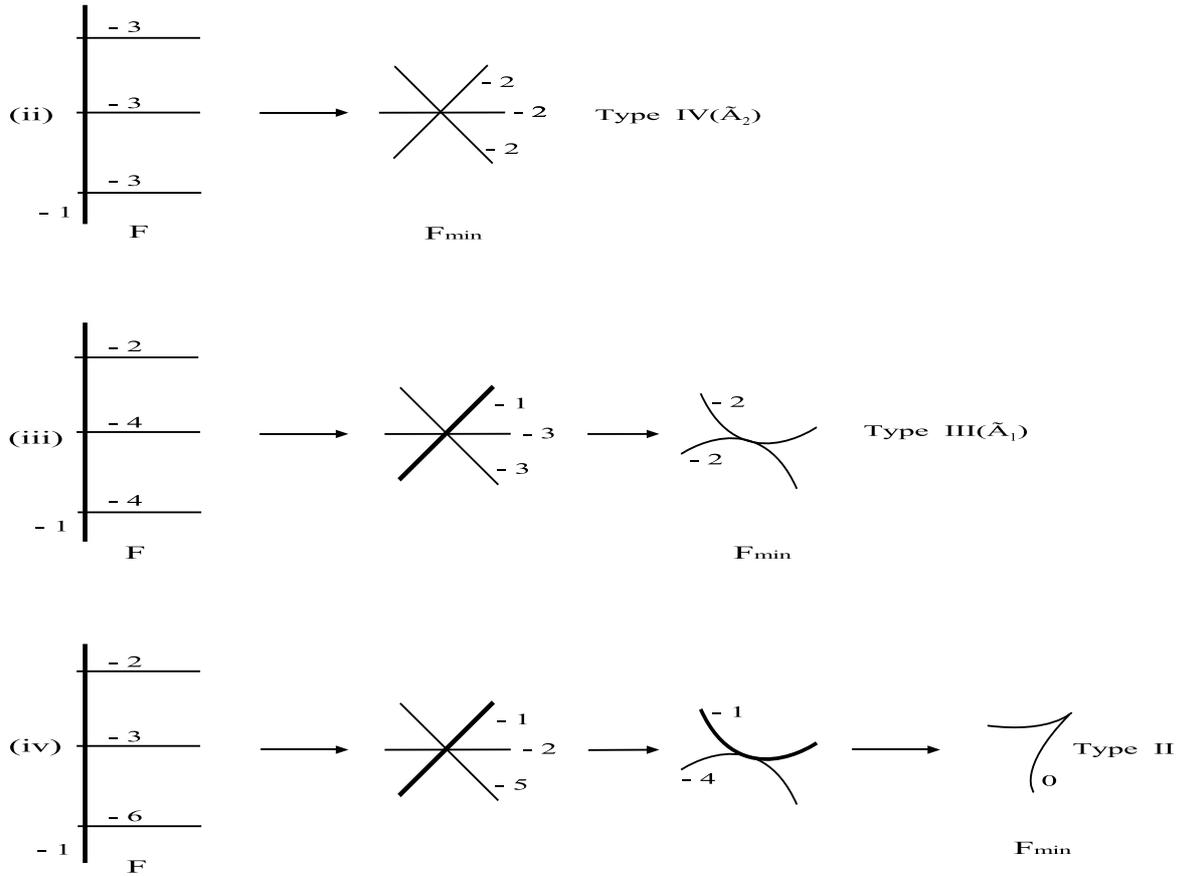}
\end{center}
\caption{$(-1)$-fibres and their minimal models in genus $1$}
\label{figure-minimal}
\end{figure}
\end{proof}

\begin{remark} \label{serrano-elliptic}
Proposition \emph{\ref{g=1}} generalizes Serrano's example of a
nonstandard elliptic isotrivial fibration having a singular fibre of
type $II$ $($see \cite[Proposition 2.5]{Se90}$)$. A strictly related
result, namely the existence of isotrivial elliptic fibrations $f
\colon X \lr \mathscr{D}$ over an open disk having the central fibre
of type
 $IV(\widetilde{A}_2)$, $III(\widetilde{A}_1)$ or $II$, appears
 in \cite[Chapter V, p. 137-138]{BPV}.
\end{remark}

\subsection{The cases $\mg=2$ and $\mg=3$}

Ogg classified in \cite{Ogg66} all singular fibres that may occur in
pencils of genus $2$ curves; in particular he showed that they are
either irreducible or belong to $44$ reducible types. In the
following proposition we classify all $(-1)$ fibres $F$ in genus $2$
and we give the corresponding type of $F_{\textrm{min}}$ according
to Ogg's classification.

\begin{proposition} \label{g=2}
There are precisely six types of $(-1)$-fibres $F$ in genus $\mg=2$. The
type of $F$, the values of $\cF$ and $\delta(F)$ and the type of
$F_{\emph{min}}$ are as in the table below.
\begin{equation*}
\begin{tabular}{c|c|c|c}
\hline
$\emph{Type of }F$ & $\cF$ & $\delta(F)$ & $\emph{Type of }F_{\emph{min}}$  \\
\hline $\big(\frac{1}{5}, \, \frac{1}{5}, \, \frac{3}{5} \big)$ &
 $2$ & $3+\frac{3}{5}$ & $\emph{Type }36$ \\
$\big(\frac{1}{5}, \, \frac{2}{5}, \, \frac{2}{5} \big)$ & $1$ &
$4+\frac{4}{5}$ & $\emph{Type }8$\\
$\big(\frac{2}{3}, \, \frac{1}{6}, \, \frac{1}{6} \big)$ & $3$ &
$3$ & $\emph{Type }34$ \\
$\big(\frac{1}{2}, \, \frac{1}{8}, \, \frac{3}{8} \big)$ & $3$ &
$3$ & $\emph{Type }1$\\
$\big(\frac{1}{2}, \, \frac{1}{5}, \, \frac{3}{10} \big)$ & $2$ &
$3+\frac{4}{5}$  & $\emph{Type }16$\\
$\big(\frac{1}{2}, \, \frac{2}{5}, \, \frac{1}{10} \big)$ & $4$ &
$2+\frac{2}{5}$  & $\emph{Irreducible}$ \\ \hline
\end{tabular}
\end{equation*}
\end{proposition}
\begin{proof}
The cyclic groups $G$ acting in genus $2$ with rational quotient and
the respective signatures are as follows (\cite[p. 252]{Br90}):
\begin{itemize}
\item[$(i)$] $G=\mathbb{Z}_2, \quad (0 \;| \; 2^6);$
\item[$(ii)$] $G=\mathbb{Z}_3, \quad (0 \, | \, 3^4);$
\item[$(iii)$] $G=\mathbb{Z}_4, \quad (0 \, | \, 2^2, 4^2);$
\item[$(iv)$] $G=\mathbb{Z}_5, \quad (0 \, | \, 5^3);$
\item[$(v)$] $G=\mathbb{Z}_6, \quad (0 \, | \, 3, 6^2);$
\item[$(vi)$] $G=\mathbb{Z}_6, \quad (0 \, | \, 2^2, 3^2);$
\item[$(vii)$] $G=\mathbb{Z}_8, \quad (0 \, | \, 2, 8^2);$
\item[$(viii)$] $G=\mathbb{Z}_{10}, \quad (0 \, | \, 2,5,10)$.
\end{itemize}
In cases $(i)$, $(ii)$, $(iii)$ and $(vi)$ we cannot have any
$(-1)$-fibre. \\
In case $(iv)$, if $F$ is a $(-1)$-fibre there are two possibilities:
\\
$(iv_a)$ $F$ is of type $\big( \frac{1}{5}, \frac{1}{5}, \frac{3}{5}
\big)$; we have $\cF=2 $ and
$\delta(F)=\frac{1}{3}(2B(\frac{1}{5})+B(\frac{3}{5}))-2=3+\frac{3}{5}$;
\\
$(iv_b)$ $F$ is of type $\big( \frac{1}{5}, \frac{2}{5}, \frac{2}{5}
\big)$; we have $\cF=1$, hence
$\delta(F)=\frac{1}{3}(B(\frac{1}{5})+2B(\frac{2}{5}))-1=4+\frac{4}{5}$.
\\
In case $(v)$ a $(-1)$-fibre is necessarily of type $\big(
\frac{2}{3}, \frac{1}{6}, \frac{1}{6} \big)$; we have $\cF=3$,
 hence
$\delta(F)=\frac{1}{3}(B(\frac{2}{3})+2B(\frac{1}{6}))-3=3$. \\
In case $(vii)$ a $(-1)$-fibre is necessarily of type $\big(
\frac{1}{2}, \frac{1}{8}, \frac{3}{8} \big)$; we have $\cF=3$,
 hence
$\delta(F)=\frac{1}{3}(B(\frac{1}{2})+B(\frac{1}{8})+B(\frac{3}{8}))-3=3$. \\
In case $(viii)$ there are again two possibilities: \\
$(viii_a)$ $F$ is of type $\big( \frac{1}{2}, \frac{1}{5},
\frac{3}{10} \big)$; we have $\cF=2$, hence
$\delta(F)=\frac{1}{3}(B(\frac{1}{2})+B(\frac{1}{5})+
B(\frac{3}{10}))-2=3+\frac{4}{5}$; \\
$(viii_b)$ $F$ is of type $\big( \frac{1}{2}, \frac{2}{5},
\frac{1}{10} \big)$; we have $\cF=4$, hence
$\delta(F)=\frac{1}{3}(B(\frac{1}{2})+B(\frac{2}{5})+
B(\frac{1}{10}))-4=2+\frac{2}{5}$. \\ \\
By looking at the classification of singular fibres in \cite{Ogg66},
one sees that the types of $F_{\textrm{min}}$ are precisely those in
our table and this completes the proof.
\end{proof}

In the same way, we can give the following list of $(-1)$-fibres in genus $3$.

\begin{proposition} \label{g=3}
There are precisely $17$ types of $(-1)$-fibres $F$ in genus $\mg=3$.
The type of $F$ and the corresponding values of $\cF$ and
 $\delta(F)$ are as in the table below.
\begin{equation*}
\begin{tabular}{c|c|c}
\hline
$\emph{Type of }F$ & $\cF$ & $\delta(F)$  \\
\hline $\big(\frac{1}{4}, \, \frac{1}{4}, \, \frac{1}{4}, \,
\frac{1}{4} \big)$ & $1$ &
$5$ \\
$\big( \frac{1}{7}, \, \frac{1}{7}, \, \frac{5}{7} \big)$ & $3$ & $4 + \frac{4}{7}$ \\
$\big( \frac{1}{7}, \, \frac{2}{7}, \, \frac{4}{7} \big)$ & $2$ & $5$ \\
$\big( \frac{1}{7}, \, \frac{3}{7}, \, \frac{3}{7} \big)$ & $1$ & $6 + \frac{6}{7} $ \\
$\big( \frac{2}{7}, \, \frac{2}{7}, \, \frac{3}{7} \big)$ & $1$ & $6+ \frac{2}{7} $ \\
$\big( \frac{1}{4}, \, \frac{1}{8}, \, \frac{5}{8} \big)$ & $2$ & $5$ \\
$\big( \frac{1}{4}, \, \frac{3}{8}, \, \frac{3}{8} \big)$ & $1$ & $5$ \\
$\big( \frac{3}{4}, \, \frac{1}{8}, \, \frac{1}{8} \big)$ & $4$ & $4$ \\
$\big( \frac{1}{3}, \, \frac{1}{9}, \, \frac{5}{9} \big)$ & $3$ & $3+\frac{8}{9} $ \\
$\big( \frac{1}{3}, \, \frac{2}{9}, \, \frac{4}{9} \big)$ & $1$ & $6+\frac{2}{9} $ \\
$\big( \frac{2}{3}, \, \frac{1}{9}, \, \frac{2}{9} \big)$ & $3$ & $4+\frac{4}{9} $ \\
$\big( \frac{1}{2}, \, \frac{1}{12}, \, \frac{5}{12} \big)$ & $4$ & $4$ \\
$\big( \frac{1}{3}, \, \frac{1}{4}, \, \frac{5}{12} \big)$ & $1$ & $4+\frac{2}{3} $ \\
$\big( \frac{2}{3}, \, \frac{1}{4}, \, \frac{1}{12} \big)$ & $4$ &  $3+ \frac{1}{3} $ \\
$\big( \frac{1}{2}, \, \frac{1}{7}, \, \frac{5}{14} \big)$ & $3$ & $3+ \frac{2}{7} $ \\
$\big( \frac{1}{2}, \, \frac{2}{7}, \, \frac{3}{14} \big)$ & $2$ & $4 + \frac{1}{7} $ \\
$\big( \frac{1}{2}, \, \frac{3}{7}, \, \frac{1}{14} \big)$ & $5$ & $3+ \frac{3}{7} $ \\
\hline
\end{tabular}
\end{equation*}
\end{proposition}
\begin{proof}
The cyclic groups $G$ acting in genus $3$ with rational quotient and
the corresponding signatures are as follows (\cite[p.
254-255]{Br90}):
\begin{itemize}
\item[$(i)$] $G=\mathbb{Z}_2, \quad (0 \;| \; 2^8);$
\item[$(ii)$] $G=\mathbb{Z}_3, \quad (0 \, | \, 3^5);$
\item[$(iii)$] $G=\mathbb{Z}_4, \quad (0 \, | \, 4^4);$
\item[$(iv)$] $G=\mathbb{Z}_4, \quad (0 \, | \, 2^3, 4^2);$
\item[$(v)$] $G=\mathbb{Z}_6, \quad (0 \, | \, 2^2, 6^2);$
\item[$(vi)$] $G=\mathbb{Z}_6, \quad (0 \, | \, 2, 3^2, 6);$
\item[$(vii)$] $G=\mathbb{Z}_7, \quad (0 \, | \, 7^3);$
\item[$(viii)$] $G=\mathbb{Z}_8, \quad (0 \, | \, 4, 8^2);$
\item[$(ix)$] $G=\mathbb{Z}_9, \quad (0 \, | \, 3, 9^2);$
\item[$(x)$] $G=\mathbb{Z}_{12}, \quad (0 \, | \, 2, 12^2);$
\item[$(xi)$] $G=\mathbb{Z}_{12}, \quad (0 \, | \, 3, 4, 12);$
\item[$(xii)$] $G=\mathbb{Z}_{14}, \quad (0 \, | \, 2, 7, 14)$.
\end{itemize}
In cases $(i)$, $(ii)$, $(iv)$, $(v)$ and $(vi)$ we cannot have any
$(-1)$-fibre, whereas the remaining possibilities give
the occurrences in the table. The details are as in the proof of
Proposition
 \ref{g=2} and they are left to the reader, who can check them by
 using the table in Appendix B.\\
\end{proof}

\newpage

\section*{Appendix B. List of cyclic quotient singularities
$x=\frac{1}{n}(1,q)$ with  $2
\leq n \leq 14$.}

\label{tabella}
$$
\begin{tabular}{|c|c|c|c|c|}

\ssi{n}{q}{n/q = [b_1, \dots ,
b_s]}{q^{\prime}}{B\big(\frac{q}{n} \big)}{h \big( \frac{q}{n} \big)}

 & & & &  \\

\ssi{2}{1}{[ 2 ]}{1}{3}{0}

\ssi{3}{1}{[ 3 ]}{1}{3+2/3}{-1/3}

\ssi{3}{2}{[ 2, 2 ]}{2}{5+1/3}{0}

\ssi{4}{1}{[ 4 ]}{1}{4+1/2}{-1}

\ssi{4}{3}{[ 2, 2, 2 ]}{3}{7+1/2}{0}

\ssi{5}{1}{[ 5 ]}{1}{5+2/5}{-9/5}

\ssi{5}{2}{[ 3, 2 ]}{3}{6}{-2/5}

\ssi{5}{4}{[ 2, 2, 2, 2 ]}{4}{9+3/5}{0}

\ssi{6}{1}{[ 6 ]}{1}{6+1/3}{-8/3}

\ssi{6}{5}{[ 2, 2, 2, 2, 2 ]}{5}{11+2/3}{0}

\ssi{7}{1}{[ 7 ]}{1}{7+2/7}{-25/7}

\ssi{7}{2}{[ 4, 2 ]}{4}{6+6/7}{-8/7}

\ssi{7}{3}{[ 3, 2, 2 ]}{5}{8+1/7}{-3/7}

\ssi{7}{6}{[ 2, 2, 2, 2, 2, 2 ]}{6}{13+5/7}{0}

\ssi{8}{1}{[ 8 ]}{1}{8+1/4}{-9/2}

\ssi{8}{3}{[ 3, 3 ]}{3}{6+3/4}{-1}

\ssi{8}{5}{[ 2, 3, 2 ]}{5}{8+1/4}{-1/2}

\ssi{8}{7}{[ 2, 2, 2, 2, 2, 2, 2 ]}{7}{15+3/4}{0}

\ssi{9}{1}{[ 9 ]}{1}{9+2/9}{-49/9}

\ssi{9}{2}{[ 5, 2 ]}{5}{7+7/9}{-2}

\ssi{9}{4}{[ 3, 2, 2, 2 ]}{7}{10+2/9}{-4/9}

\ssi{9}{8}{[ 2, 2, 2, 2, 2, 2, 2, 2 ]}{8}{17+7/9}{0}

\ssi{10}{1}{[ 10 ]}{1}{10+1/5}{-32/5}

\ssi{10}{3}{[ 4, 2, 2 ]}{7}{9}{-6/5}

\ssi{11}{1}{[ 11 ]}{1}{11+2/11}{-81/11}

\ssi{11}{2}{[ 6, 2 ]}{6}{8+8/11}{-32/11}

\ssi{11}{3}{[ 4, 3 ]}{4}{7+7/11}{-20/11}

\ssi{11}{5}{[ 3, 2, 2, 2, 2 ]}{9}{12+3/11}{-5/11}

\ssi{11}{7}{[ 2, 3, 2, 2 ]}{8}{10+4/11}{-6/11}

\ssi{12}{1}{[ 12 ]}{1}{12+1/6}{-25/3}

\ssi{12}{5}{[ 3, 2, 3 ]}{5}{8+5/6}{-1}

\ssi{12}{7}{[ 2, 4, 2 ]}{7}{9+1/6}{-4/3}

\ssi{13}{1}{[ 13 ]}{1}{13+2/13}{-121/13}

\ssi{13}{2}{[ 7, 2 ]}{7}{9+9/13}{-50/13}

\ssi{13}{3}{[ 5, 2, 2 ]}{9}{9+12/13}{-27/13}

\ssi{13}{4}{[ 4, 2, 2, 2 ]}{10}{11+1/13}{-16/13}

\ssi{13}{5}{[ 3, 3, 2 ]}{8}{9}{-15/13}

\ssi{13}{6}{[ 3, 2, 2, 2, 2, 2 ]}{11}{14+4/13}{-6/13}

\ssi{14}{1}{[ 14 ]}{1}{14+1/7}{-72/7}

\ssi{14}{3}{[ 5, 3 ]}{5}{8+4/7}{-19/7}

\ssi{14}{9}{[ 2, 3, 2, 2, 2 ]}{11}{12+3/7}{-4/7}






\hline \hline
\end{tabular}
$$


\newpage





\begin{thebibliography}{999999}







\bibitem[BaCaGr08]{BaCaGr08}
I. Bauer, F. Catanese, F. Grunewald: The classification of surfaces
with $p_g=q=0$ isogenous to a product of curves, \emph{Pure Appl.
Math. Q.} $\boldsymbol{4}$ (2008),  no. 2, part 1, 547--5861.

\bibitem[BaCaGrPi08]{BaCaGrPi08}
I. Bauer, F. Catanese, F. Grunewald, R. Pignatelli: Quotient of a
product of curves by a finite group and their fundamental groups,
e-print $\mathbf{arXiv:0809.3420}$ (2008).



\bibitem[BPV84]{BPV}
W. Barth, C. Peters, A. Van de Ven: \emph{Compact Complex Surfaces},
Springer-Verlag 1984.

\bibitem[Bar99]{Bar99}
R. Barlow: Zero-cycles on Mumford's surface, \emph{Math. Proc. Camb.
Phil. Soc.} $\boldsymbol{126}$ (1999), 505-510.




\bibitem[Be96]{Be2}
A. Beauville: \emph{Complex algebraic surfaces}, Cambridge
University Press 1996.




\bibitem[Bre00]{Bre00}
T. Breuer: \emph{Characters and Automorphism groups of Compact
Riemann Surfaces},
 Cambridge University Press 2000.

\bibitem[Br90]{Br90}
S. A. Broughton: Classifying finite group actions on surfaces of low
genus, \emph{J. of Pure and Applied Algebra} $\boldsymbol{69}$
(1990), 233-270.





\bibitem[Ca00]{Ca00}
F. Catanese: Fibred surfaces, varieties isogenous to a product and
related moduli spaces, \emph{American Journal of Mathematics}
$\boldsymbol{122}$ (2000), 1-44.








\bibitem[CarPol07]{CarPol07}
G. Carnovale, F. Polizzi: The classification of surfaces of
general type with $p_g=q=1$ isogenous to a product,  \emph{Adv.
Geom.} $\boldsymbol{9}$ (2009), 233-256.



\bibitem[CCPW]{CCPW}
J. H. Conway, R. T. Curtis, R. A. Parker, R. A. Wilson: \emph{Atlas of finite groups},
Oxford University Press 1985.






\bibitem[Deb81]{Deb81}
O. Debarre: Inegalit{\'es} num{\'e}riques pour les surfaces de type
g{\'e}n{\'e}rale, \emph{Bull. Soc. Math. de France}
 $\boldsymbol{110}$ (1982), 319-346.




\bibitem[FK92]{FK92}
H. M. Farkas, I. Kra: \emph{Riemann Surfaces}, Graduate Texts in
Mathematics $\boldsymbol{71}$, $2^{nd}$ Edition, Springer-Verlag
1992.



\bibitem[Fre71]{Fre71}
E. Freitag: Uber die Struktur der Funktionenk$\ddot{\textrm{o}}$rper
zu hyperabelschen Gruppen I, \emph{J. Reine. Angew. Math.}
$\boldsymbol{247}$ (1971), 97-117.




\bibitem[GAP4]{GAP4}
The GAP~Group, \emph{GAP -- Groups, Algorithms, and Programming,
  Version 4.4}; 2006, http://www.gap-system.org.







\bibitem[H71]{H71} W. J. Harvey: On the branch loci in
Teichm{\"u}ller space, \emph{Trans. Amer. Mat. Soc.}
$\boldsymbol{153}$ (1971), 387-399.











\bibitem[JS87]{JS87}
G. A. Jones, D. Singerman: \emph{Complex Functions}, Cambridge
University Press 1987.











\bibitem[Lau71]{Lau71}
H. B. Laufer: \emph{Normal two-dimensional singularities}, Annals of
Mathematics Studies $\boldsymbol{71}$,
 Princeton University Press 1971.







\bibitem[MiPol08]{MiPol08}
E. Mistretta, F. Polizzi: Standard isotrivial fibrations with
$p_g=q=1. \;II$, \emph{J. Pure Appl. Algebra} $\boldsymbol{214}$
(2010), 344-369.




\bibitem[NePo08]{NePo08}
A. Némethi and P. Popescu-Pampu: On the Milnor fibers of cyclic
quotient singularities, e-print $\mathbf{arXiv:0805.3449v1}$ (2008).

\bibitem[Ogg66]{Ogg66}
A. P. Ogg: On pencils of curves of genus $2$, \emph{Topology}
$\boldsymbol{5}$ (1966), 355-362.

\bibitem[OW77]{OW77}
P. Orlik, P. Wagreich: Algebraic surfaces with $k^*$-action,
\emph{Acta Math.} $\boldsymbol{138}$ (1977), 43-81







\bibitem[Pol07]{Pol07}
F. Polizzi: Standard isotrivial fibrations with $p_g=q=1$, \emph{J. Algebra}
$\boldsymbol{321}$ (2009), 1600-1631.



\bibitem[Rie74]{Rie74}
O. Riemenschneider: Deformationen von
quotientensingularit$\ddot{\textrm{a}}$ten (nach Zyklischen
Gruppen), \emph{Math. Ann.} $\boldsymbol{209}$, 211-248 (1974).







\bibitem[Se90]{Se90}
F. Serrano: Fibrations on algebraic surfaces, \emph{Geometry of
Complex Projective Varieties $($Cetraro 1990$)$}, A. Lanteri, M.
Palleschi, D. C. Struppa eds., Mediterranean Press (1993), 291-300.



\bibitem[Se92]{Se92}
F. Serrano: Elliptic surfaces with an ample divisor of genus two, \emph{Pacific J.
Math.} $\boldsymbol{152}$ (1992), 187-199.


\bibitem[Se96]{Se96}
F. Serrano: Isotrivial fibred surfaces, \emph{Annali di Matematica
pura e applicata}, vol. CLXXI (1996), 63-81.

\bibitem[Tan96]{Tan96}
S. L. Tan: On the invariant of base changes of pencils of curves, II, \emph{Math. Z.}
 $\boldsymbol{222}$ (1996), 655-676.









\end{thebibliography}
\end{document}